\documentclass[a4paper,10pt,reqno]{amsart}

\usepackage[utf8]{inputenc}
\usepackage[T1]{fontenc}
\usepackage{amsthm}
\usepackage{amsmath}
\usepackage{amssymb}
\usepackage[inline]{enumitem}
\usepackage{comment}
\usepackage{hyperref}
\usepackage{fancyhdr}
\usepackage{mathrsfs}
\usepackage{stmaryrd}
\usepackage[normalem]{ulem}
\usepackage{xcolor}
\usepackage{nicefrac}

\theoremstyle{definition}

\newtheorem{theorem}{Theorem}[section]
\newtheorem{lemma}[theorem]{Lemma}
\newtheorem{proposition}[theorem]{Proposition}
\newtheorem{corollary}[theorem]{Corollary}

\theoremstyle{definition}

\newtheorem{example}[theorem]{Example}
\newtheorem{remark}[theorem]{Remark}

\newcommand{\evid}[1]{\textsf{#1}}

\providecommand\llb{\llbracket}
\providecommand\rrb{\rrbracket}
\providecommand\spa{{\rm Span}}
\providecommand\I{{\mathcal{I}}}
\providecommand\Mon{{\mathcal{M}{\rm on}}}

\DeclareMathOperator{\mdeg}{\mathsf{m-deg}}
\DeclareMathOperator{\ini}{\mathsf{in}}
\DeclareMathOperator{\Min}{\mathrm{Min}}

\title{On the arithmetic of polynomial ideals}

\author{Nikola Bogdanovic}
\address{Department of Mathematics and Scientific Computing, University of Graz | Heinrichstrasse 36/III, 8010 Graz, Austria}
\email{nikola.bogdanovic@uni-graz.at}

\author{Laura Cossu}
\address{Department of Mathematics and Computer Science, University of Cagliari | Via Ospedale 72, 09124 Cagliari, Italy}
\email{laura.cossu3@unica.it}

\author{Azeem Khadam}
\address{Department of Mathematics and Scientific Computing, University of Graz | Heinrichstrasse 36/III, 8010 Graz, Austria}
\email{azeemkhadam@gmail.com}



\thanks{This research was funded in part by the Austrian Science Fund (FWF) under the projects 10.55776/DOC183 and 10.55776/P\-AT\-975\-6623. N.~B.~also acknowledges funding from the Department of Economy, Tourism, Science and Research of the Federal Province of Styria (Austria). L.~C.~acknowledges the European Union's Horizon 2020 program (Marie Sklodowska-Curie grant 101021791) for financial support during the early stages of the development of this manuscript, and the research group GNSAGA of the Italian Mathematics Research Institute (INdAM). 
The authors are grateful to the anonymous referee for their careful reading and insightful comments, which contributed to improving the presentation of the paper. In particular, they thank the referee for suggesting the observation preceding Lemma 3.8 and Proposition 4.5.}

\begin{document}
\begin{abstract}
This paper investigates atomic factorizations in the monoid $\I(R)$ of nonzero ideals of a multivariate polynomial ring $R$, under ideal multiplication. Building on recent advances in factorization theory for unit-cancellative monoids, we extend techniques from \cite{Ge-Kh21} to construct new families of atoms in $\I(R)$, leading to a deeper understanding of its arithmetic. We further analyze the submonoid $\Mon(R)$ of nonzero monomial ideals, deriving arithmetic properties and computing sets of lengths for specific classes of ideals. The results advance the extensive study of ideal monoids within a classical algebraic framework.
\end{abstract}
\maketitle

\section{Introduction}

By the Fundamental Theorem of Arithmetic, every positive integer admits a unique factorization into prime numbers, up to units. This result, first formulated in Euclid’s {\it Elements}, was later found not to extend to more general settings. In particular, when considering the ring of integers of a number field (a number ring, for short), factorizations into prime elements may fail to be unique. Initially regarded as a pathological exception, starting with Narkiewicz’s work in the 1960s, this phenomenon became a subject of systematic study. The resulting field, now known as {\it factorization theory}, investigates the structure of non-unique factorizations. Its main object of study is the following notion, suitably generalizing that of factorization into primes in (multiplicative monoids of) number rings: an \emph{(atomic) factorization} of an element of a commutative cancellative monoid is, up to units, a finite, unordered sequence of \emph{atoms} (i.e., non-unit elements which are not the product of two non-units) whose product is equal to the element. We refer the reader to the monograph \cite{Ge-HK} for a comprehensive overview of the subject.

In recent years, while the classical theory of factorization has continued to thrive (see, for instance, \cite{C-C-G-S21, Co-Ha18, Co-Go19, Go18a,Go19a,Got-Li-2023, Ger-Got-2025}),
several attempts have been made to expand its scope, considering non-commutative (see \cite{Ba-Sm15, Bell17, Sm-Vo19a, Na-Sm23}) and non-cancellative monoids (see \cite{An-Tr21, Co-Tr24, Tr21(b)}), even multiplicative monoids of rings with zero divisors, as in \cite{AnVL96, AnVL97, Ba-Po26}. In the latter direction, the definition of \emph{unit-cancellativity} was introduced in \cite{Fan17} as a weakening of the usual cancellativity hypothesis, and a substantial portion of the classical theory has since been developed in this more general setting; we refer to \cite{Ge-Zh20} for a gentle introduction to the topic.

We aim to study atomic factorizations in the (commutative, unit-cancellative) monoid of nonzero ideals of the polynomial ring $K[X_1,\dots,X_N]$, where $K$ is an arbitrary field, under ideal multiplication. We assume throughout the paper that $N\ge 2$, as $K[X]$ is a PID and the factorizations of its ideals correspond to the factorizations of their generating polynomials.

Our work contributes to the extensive literature on monoids of ideals: we refer to \cite{Hal-Ko98} and the monograph \cite{Fo-Ho-Lu13} for an overview, and to \cite{Rein12,An-Ju19,Kl-Om20,Ol-Re20} for more recent developments. A common approach in the literature is to focus on specific classes of ideals within a given ring, or semigroup ideals within a given monoid. Among these, divisorial ideals are perhaps the most prominent example. By suitably choosing a ring (or monoid) and a class of its ideals, one can obtain a fairly well-behaved monoid of ideals, and in some cases one is able to completely determine its structure. As an example, we recall two well-known results: a domain $R$ is Dedekind if and only if all its nonzero ideals are invertible, and this happens only if the monoid of all nonzero ideals is factorial, and a domain $R$ is Krull if and only if the monoid of its invertible divisorial ideals is factorial.

Compared to the \lq\lq classical\rq\rq\ examples of monoids of ideals, the structure of atomic factorizations in the monoid $\I(R)$ consisting of all nonzero ideals of a ring $R$, remains poorly understood in general. This is true even in seemingly simple cases, such as when $R$ is a polynomial ring over a field, where surprisingly little is currently known.

In the paper \cite{Ge-Kh21}, a first step was taken in studying the arithmetic properties of $\I(R)$, with $R$ a multivariate polynomial ring. In particular, the results collected in \cite[Theorem 5.1]{Ge-Kh21} reveal a noticeable similarity between the arithmetic of $\I(R)$ and that of Krull monoids with infinite class group and prime divisors in each class, despite $\I(R)$ not even being transfer Krull. More specifically, $\I(R)$ is fully elastic and each of its unions of sets of lengths is the whole set $\mathbb N_{\ge 2}$. 

Building on these initial results, we refine the proof techniques and adapt them to broader contexts. Following a review of the necessary preliminaries in Section \ref{sec: preliminaries}, we extend the methods introduced in \cite{Ge-Kh21} to find several new families of atoms in $\I(R)$. In particular, in Section \ref{subsec: sum-free} we highlight an interesting class built from sum-free subsets of $\mathbb{N}$, which includes some previously known examples.

Motivated by the problem of computing the length sets of certain ideals (see Theorem \ref{thm: I_B atom}), in Section \ref{sec: atoms of Mon(R)} we focus on the submonoid $\Mon(R)\subseteq\I(R)$ of nonzero \emph{monomial ideals} of $R$ (i.e., ideals generated by monomials).

We first provide some arithmetic results in the spirit of \cite[Theorem 5.1]{Ge-Kh21}, and then we determine several classes of atoms in this monoid, allowing us to compute the set of lengths of the families of ideals introduced in Theorem \ref{thm: I_B atom}. The main technical ingredient that makes the study of factorizations in $\Mon(R)$ easier than in $\I(R)$ is Lemma \ref{lem: mon_gen} (see also Remark \ref{rem: mon min gen}), greatly restricting the possible generators of the monomial ideals occurring in a given factorization. 

We conclude the paper by outlining, in Section \ref{sec: conclusion}, possible directions for future research on atomic factorizations in the monoids $\I(R)$ and $\Mon(R)$, with the present work serving as a foundation.

\section{Notation and preliminary results}\label{sec: preliminaries}
Throughout the paper, $\mathbb{N}$ and $\mathbb{N}^+$ denote the sets of nonnegative and positive integers, respectively. For $a,b\in \mathbb N \cup \{\infty\}$ we let $\llb a, b\rrb := \allowbreak \{x\in \mathbb{N} \colon a\leq x\leq b\}$ be the discrete interval from $a$ to $b$. For sets $A,B\subseteq\mathbb N$, we denote by $A+B:=\{a+b: a\in A,b\in B\}$ their sumset and by $\Delta(A)\subseteq\mathbb N^+$ the set of distances of $A$, i.e., the set of all $d\in\mathbb N^+$ such that $A\cap \llb a,a+d\rrb=\{a,a+d\}$ for some $a\in A$. For every nonempty set $A\subseteq\mathbb N^+$, we denote by $\rho(A):={\sup A}/{\min A}\in\mathbb Q_{\ge 1}\cup\{\infty\}$ its elasticity; by convention, we set $\rho(\emptyset)=\rho(\{0\})=1$.

\subsection{Factorization theory}\label{subsec: factorization theory}
In order to make the paper as self-contained as possible, we recall the basic notions of factorization theory. For further details, examples, and references, the reader can consult the already mentioned monograph \cite{Ge-HK}.

Let $H$ be a commutative semigroup with an identity element $1_H$, written multiplicatively. We denote by $H^\times$ the subgroup of invertible elements of $H$, also called \emph{units} of $H$. If $H^\times=\{1_H\}$, we say that $H$ is \emph{reduced}. An element $a\in H$ is \emph{cancellative} if $b,c\in H$ and $ab=ac$ implies $b=c$, while it is \emph{unit-cancellative} if $b\in H$ and $a=ab$ implies $b\in H^\times$. Moreover, we say that $H$ is \emph{cancellative} (respectively, \emph{unit-cancellative}) if all its elements are cancellative (respectively, unit-cancellative).

In this paper, a \emph{monoid} is always a commutative unit-cancellative semigroup with an identity element. We use the term \emph{submonoid} only for subsemigroups that contain the identity element and are themselves unit-cancellative. For example, a subsemigroup of a reduced monoid $H$ that contains $1_H$ is a submonoid of $H$. A monoid $H$ is cancellative if and only if it has a quotient group, which we denote by $\mathbf q(H)$.

Let $H$ be a cancellative monoid. An \emph{ideal} of $H$ is a subset $J\subseteq H$ such that, for any $a\in J$, $aH\subseteq J$; they are also commonly referred to as \emph{$s$-ideals} (\emph{semigroup} ideals), especially when $H$ also carries a ring structure, to distinguish them from usual ring ideals. For ideals $J,K$ of $H$, denote by $(J:K)$ the \emph{ideal quotient} $\{x\in\mathbf q(H): xK\subseteq J\}$. An ideal $J$ of $H$ is \emph{divisorial} if $J=(H:(H:J))$. We say that $H$ is a \emph{Mori} monoid if it satisfies the ascending chain condition on divisorial ideals. Furthermore, we denote by $\hat H:=\{x\in\mathbf q(H): \text{ there exists } a\in H \text{ such that } ax^n\in H \text{ for all }n\in\mathbb N\}$ the \emph{complete integral closure} of $H$, and we say that $H$ is \emph{completely integrally closed} if $H=\hat H$. Finally, we say that $H$ is a \emph{Krull} monoid if it is Mori and completely integrally closed.

Let $H$ be a monoid, not necessarily cancellative. A submonoid $S\subseteq H$ is \emph{divisor-closed} if, whenever $a\in S,b\in H$ and $b\mid a$, we have $b\in S$. For a subset $A\subseteq H$, we denote by $\llb A\rrb$ the smallest divisor-closed submonoid containing $A$. A monoid homomorphism $\varphi: H\to H'$, where $H'$ is cancellative, is a \emph{divisor homomorphism} if $a,b\in H$ and $\varphi(a)\mid\varphi(b)$ in $H'$ implies that $a\mid b$ in $H$.

An element $a\in H$ is \emph{prime} if $a\notin H^\times$ and $a\mid bc$ with $b,c\in H$ implies $a\mid b$ or $a\mid c$. An element $a\in H$ is an \emph{atom} if $a\notin H^\times$ and $a=bc$ with $b,c\in H$ implies $b\in H^\times$ or $c\in H^\times$; the set of all atoms of $H$ is denoted by $\mathscr A(H)$. A \emph{factorization} of an element $a\in H$ is a formal product $z=u_1\cdots u_n$ of atoms of $H$, where $n\in\mathbb N$, defined up to permutations and replacement of atoms by (unit-)associates, such that $a=\varepsilon u_1\cdots u_n$ for some $\varepsilon\in H^\times$. Its \emph{length} is $\lvert z\rvert:=n$. In particular, every unit has the unique empty factorization, which has length $0$. This is the standard convention of \cite{Ge-HK}. We denote by
\begin{itemize}
    \item $\mathsf Z_H(a)=\mathsf Z(a)$ the \emph{set of factorizations} of $a\in H$;
    \item $\mathsf L_H(a)=\mathsf L(a):=\{\lvert z\rvert:z\in\mathsf Z(a)\}\subseteq\mathbb N$ the \emph{set of lengths} of $a$;
    \item $\mathcal L(H):=\{\mathsf L(a):a\in H\}$ the \emph{system of sets of lengths} of $H$;
    \item $\mathcal U_k(H):=\bigcup_{k\in L\in \mathcal L(H)} L$ the \emph{union of sets of lengths} of $H$ containing $k\in\mathbb N$;
    \item $\rho(H):=\sup\{\rho(L)\colon L\in\mathcal L(H)\}$ the \emph{elasticity} of $H$.
\end{itemize}
Note that, if $S\subseteq H$ is a divisor-closed submonoid and $a\in S$, then $\mathsf Z_S(a)=\mathsf Z_H(a)$ and $\mathsf L_S(a)=\mathsf L_H(a)$. We say that $H$ is
\begin{itemize}
    \item \emph{atomic} if $\mathsf L(a)$ is nonempty for every $a\in H$;
    \item a \emph{$\mathsf{BF}$-monoid} if $\mathsf L(a)$ is finite and nonempty for every $a\in H$;
    \item an \emph{$\mathsf{FF}$-monoid} if $\mathsf Z(a)$ is finite and nonempty for every $a\in H$;
    \item \emph{half-factorial} if $\lvert\mathsf L(a)\rvert = 1$ for every $a\in H$;
    \item \emph{factorial} if $\lvert\mathsf Z(a)\rvert=1$ for every $a\in H$;
    \item \emph{locally finitely generated} if ${\llb a\rrb}$ is finitely generated up to units, for every $a\in H$;
    \item \emph{fully elastic} if, for any $q\in\mathbb Q$ with $1<q<\rho(H)$, there is an $L\in\mathcal L(H)$ such that $\rho(L)=q$.
\end{itemize}
Every Mori monoid is a $\mathsf{BF}$-monoid (\cite[Theorem 2.2.5.1]{Ge-HK}), Krull monoids are locally finitely generated and locally finitely generated monoids are in turn $\mathsf{FF}$-monoids (\cite[Prop.~2.7.8]{Ge-HK}). We note that sets of elasticities of locally finitely generated monoids have been studied in great detail in \cite{Zhong19}. 

A monoid homomorphism $\theta: H\to K$ is called a \emph{transfer homomorphism} if $K=\theta(H)K^\times$, $\theta^{-1}(K^\times)=H^\times$, and, whenever $a\in H$, $\beta,\gamma\in K$ are such that $\theta(a)=\beta\gamma$, there exist $b,c\in H$ such that $a=bc$, $\theta(b)\in \beta K^\times$ and $\theta(c)\in \gamma K^\times$. In Section \ref{sec: atoms of Mon(R)} we will use the following elementary result.

\begin{lemma}\label{lem: transfer Krull}
    If $\theta:H\to K$ is a transfer homomorphism, then $\mathsf L_H(a)=\mathsf L_K(\theta(a))$ for any $a\in H$. In particular, an element $a\in H$ is an atom of $H$ if and only if $\theta(a)$ is an atom of $K$.
\end{lemma}
A monoid $H$ is \emph{transfer Krull} if there exist a Krull monoid $K$ and a transfer homomorphism $\theta:H\to K$. Clearly, any Krull monoid is transfer Krull, since the identity is a transfer homomorphism. Conversely, by Lemma \ref{lem: transfer Krull} all transfer Krull monoids are $\mathsf{BF}$-monoids, but they need neither be Mori nor completely integrally closed, as noted in \cite[Example 5.4]{Ge-Zh20}.  All transfer Krull monoids are fully elastic, by \cite[Theorem 3.1]{Ge-Zh19}. Some notable examples of monoids which are \emph{not} transfer Krull are the integer-valued polynomials, by \cite[Remark 12]{Sophie13}, and $\mathcal P_{\rm fin}(\mathbb N)$, by \cite[Prop.~4.12]{Fan-Tr18}.

\subsection{Multivariate polynomial rings and their ideals}\label{subsec: polynomial rings and ideals}
Let $D$ be an integral domain with quotient field $K$, and let $R=D[X_1,\dots,X_N]$ and $S=K[X_1,\dots, X_N]$ be the polynomial rings in $N\ge 2$ indeterminates $X_1, \dots, X_N$ over $D$ and $K$ respectively. 
For $m=(m_1,\dots,m_N)\in\mathbb{N}^N$, let $\mathbf{X}^m:=X_1^{m_1}\cdots X_N^{m_N}$ and set, for every $t\ge 0$,
\[R_t:=\Biggl\{\sum_{m\in\mathbb{N}^N}r_m\mathbf{X}^m\colon r_m\in D\text{ and } m_1+\cdots+m_N=t\Biggr\},\]
\[S_t:=\Biggl\{\sum_{m\in\mathbb{N}^N}q_m\mathbf{X}^m\colon q_m\in K\text{ and } m_1+\cdots+m_N=t\Biggr\}.\]
Then \[R=\bigoplus_{t\ge 0} R_t \text{ and }S=\bigoplus_{t\ge 0} S_t\] are graded rings, equipped with the {\it standard grading}, so that $\deg X_1=\dots=\deg X_N=1$. It is clear that $R_t\subseteq S_t$ for every $t$, that $R_0=D$ and $S_0=K$, and that every $f\in S$ can be written uniquely as $f=\sum_{t\ge 0}f_t$, with $f_t\in S_t$ for every $t$.

We can then define the \evid{min-degree} of a nonzero polynomial $f\in S$ as the smallest nonnegative integer $d$ such that $f_d$ is nonzero, and we denote it as $\mdeg(f)$. We set $\mdeg(0):=+\infty$. For all $f,g\in S\setminus\{0\}$, the following properties hold:

\begin{enumerate}[label=\textup{(\alph{*})}]
    \item $\mdeg(fg)=\mdeg(f)+\mdeg(g)$;
    \item $\mdeg(f+g)\ge\min\{\mdeg(f),\mdeg(g)\}$, with equality if $\mdeg(f)\ne\mdeg(g)$.
\end{enumerate}

Let $I\subseteq R$ be an ideal of $R$. We define the \evid{min-degree} $\mdeg(I)$ of $I$ as the smallest nonnegative integer $d$ such that $I$ contains a polynomial whose min-degree is $d$. The min-degree of the zero ideal is set equal to $+\infty$.

{
\begin{remark}\label{rem: min-deg of fg ideal}
 It is easy to see from the properties of the min-degree, that if $I$ is a nonzero finitely generated ideal of $R$, i.e., $I=\langle g_1, \dots, g_n\rangle$ with $n\ge 1$ and $g_1,\dots, g_n\in R$, then 
 \[\mdeg(I)=\min\{\mdeg(g_i)\colon i\in \llb 1,n\rrb\}.\]
\end{remark}
}

Let $\I(R)$ be the semigroup of nonzero (ring) ideals of $R$, with the usual ideal multiplication. This semigroup is commutative, has an identity element ($R$ itself) and it is reduced. Moreover, if $R$ satisfies the \emph{Krull Intersection Theorem}, i.e., if $\bigcap_{n\ge 0}I^n = \{0\}$ for every $I\in\I(R)\setminus\{R\}$ (this holds, for example, if $D$ is noetherian), then $\I(R)$ is also unit-cancellative by \cite[Prop.~2.1(4)]{Ge-Kh21}, and thus a monoid in the sense specified in Section \ref{subsec: factorization theory}.

\begin{remark}\label{rmk: cancellative I(R)}
    It is of interest to understand how far $\I(R)$ is from being cancellative. We recall here a remarkable characterization of the cancellative elements of $\I(R)$ due to Anderson and Roitman, valid for \emph{any} commutative ring $R$, not necessarily even a domain: an ideal $I$ of $R$ is cancellative if and only if it is a regular principal ideal at each localization at a maximal ideal \cite{An-Ro97}. If we consider $R=D[X_1,\dots,X_N]$, with $D$ a noetherian domain and $N\ge 2$, it is easy to check that the cancellative elements of $\I(R)$ are precisely the principal ideals of $R$.
\end{remark}

It was proved in Lemma 5.5 of \cite{Ge-Kh21} that the map $\mdeg\colon \I(R)\to \mathbb{N } \text{ defined by } I\mapsto \mdeg(I)$
is a semigroup homomorphism, i.e., for every $I,J\in\I(R)$, 
\begin{equation}\label{eq: additivity min-deg}
\mdeg(IJ)=\mdeg(I)+\mdeg(J).
\end{equation}

For a nonzero ideal $I$ of $R$ and a nonnegative integer $i$, we let \[I[i]:=\{f_i\colon f\in I\}\] be the set of homogeneous components of degree $i$ of the elements of $I$, which is a $D$-module contained in $R_i$. For $I,J\in \I(R)$ and $i,j\in \mathbb{N}$, \[I[i]\cdot J[j]:=\left\{\sum_{s=1}^k x_sy_s\colon k\ge 1, x_s\in I[i], y_s\in J[j] \right\}\] is also a $D$-module.
Moreover, we let \[I_K[i]:=\spa_K\{I[i]\}\] be the $K$-linear span of the elements of $I[i]$. It is clear that $I_K[i]$ is a finite dimensional $K$-vector space contained in $S_i$.

\begin{remark}\label{rem: span of a fg ideal}
If $I=\langle g_1, \dots, g_n\rangle$ is a finitely generated nonzero ideal of $R$ and $\mdeg(I)=m\in \mathbb{N}$, then 
 \[I_K[m]=\spa_K\{(g_i)_m \colon \mdeg(g_i)=m\}.\]
\end{remark}

\begin{lemma}[Lemma 5.6 of \cite{Ge-Kh21}]\label{lem_5.6}
Let $I,J\in \I(R)$ and set $\mdeg(I)=d$, $\mdeg(J)=e$ and $\mdeg(IJ)=m$. The following hold:

\vspace{.2cm}
\begin{enumerate*}[label=\textup{(\roman{*})}, mode=unboxed]
    \item\label{lem_5.6(i)} $(IJ)[m]=I[d]\cdot J[e]$;
\end{enumerate*}

\vspace{.2cm}
\begin{enumerate*}[label=\textup{(\roman{*})}, resume, mode=unboxed]
    \item\label{lem_5.6(ii)} $(IJ)_K[m]=I_K[d]\cdot J_K[e]$.
\end{enumerate*}
\end{lemma}

Throughout the rest of the paper, we will consider ideals of $R$ generated by polynomials in two indeterminates, say $X_1$ and $X_2$. To slightly ease the notation, we will rename these indeterminates to $X$ and $Y$ respectively.

Finally, let  $\leq$ be the \evid{lexicographic order} on monomials of $S$ defined by $X>Y>X_3> \dots>X_N>1$. We denote by $\ini(f)$ the \emph{initial monomial} of a nonzero $f\in S$ with respect to $\leq$, that is, the largest monomial with respect to the lexicographic order occurring in $f$ with a nonzero coefficient. Following \cite{En-Her}, we set $\ini(0)=0$.

\begin{remark}\label{rem: XY}
    In the notation of Lemma \ref{lem_5.6}, if we further assume $(IJ)_K[m]\subseteq K[X,Y]$, then for every $f\in I_K[d]$ and every $g\in J_K[e]$, we have $\ini(f),\,\ini(g)\in K[X,Y]$. In fact, if there exists a nonzero $f\in I_K[d]$ such that $X_i$ divides $\ini(f)$ for some $i\in \llb 3,N\rrb$, then for any nonzero $g\in J_K[e]$ (one always exists as $J$ is a nonzero ideal) we have $\ini(fg)=\ini(f)\ini(g)\not\in K[X,Y]$, and so $fg\in (IJ)_K[m]\setminus K[X,Y]$, contradicting the assumption.
\end{remark}

\subsection{A link between polynomial ideals and power monoids}\label{subsec: power monoids} Keeping the notation of the previous subsection, we let $R=D[X_1,\dots,X_N]$, with $D$ an integral domain, be a polynomial ring in $N\ge 2$ indeterminates. As mentioned in the introduction, the study of $\I(R)$ from the perspective of factorization theory was initiated in \cite[Sect.~5]{Ge-Kh21}, where it was shown that when $\I(R)$ is a $\mathsf{BF}$-monoid, it is fully elastic and each of its unions of sets of lengths equals $\mathbb N_{\ge 2}$ \cite[Theorem 5.1]{Ge-Kh21}.

The proof of the main theorem relies on the construction of several families of atoms in $\I(R)$:
\[\mathfrak{b}_i(X,Y):=\langle X^i, Y^i\rangle\ {\rm for}\ i\in\mathbb{N}^+,\]
\[\mathfrak{c}_{2i+1}(X,Y):=\langle X^{2i+1-j}Y^{j}\colon \ j\in\{0\}\cup (2\llb 0,i\rrb+\{1\})\rangle\ {\rm for}\ i\in\mathbb{N}^+,\]
\[\mathfrak{c}_{2i}(X,Y):=\langle X^{2i-j}Y^j\colon j\in \{1\}\cup 2\llb 0,i\rrb\rangle\ {\rm for}\ i\in\mathbb{N}_{\ge 3}.\]
These families of atoms share a common structure: as explained in \cite[Example 5.13]{Ge-Kh21}, they correspond to families of atoms in the monoid of finite nonempty subsets of $\mathbb{N}$ endowed with set-wise addition, known as the \emph{(finitary) power monoid} of $\mathbb{N}$ and denoted by $\mathcal{P}_{{\rm fin}}(\mathbb{N})$. More precisely, $\I(R)$ has a submonoid which is isomorphic to $\mathcal{P}_{{\rm fin}}(\mathbb{N})$. Indeed, the map 
\begin{equation}\label{eq: map}
    \Phi\colon\ \mathcal{P}_{{\rm fin}}(\mathbb{N}) \to \I(R), \; \ A\mapsto I_A
\end{equation}
that sends a set $A\in\mathcal{P}_{{\rm fin}}(\mathbb{N})$, say $A=\{n_1,\dots, n_\ell\}$ with $\ell\ge 1$ and $0\le n_1<\dots <n_\ell$, to the ideal
\[I_A:=\langle X^{n_\ell-n_1}Y^{n_1},X^{n_\ell-n_2}Y^{n_2}, \dots, X^{n_\ell-n_{\ell-1}}Y^{n_{\ell-1}},Y^{n_\ell}\rangle,\]
is readily seen to be an injective monoid homomorphism. The image of this map, denoted by $\mathcal{M}_2(X,Y)$, consists precisely of those ideals which can be generated by $Y^m$ and a (possibly empty) set of monomials in $X$ and $Y$ of degree $m$, for some $m\in\mathbb{N}$. In particular, $\mathcal{M}_2(X,Y)\subseteq\Mon(R)\subseteq\I(R)$, where $\Mon(R)$ is the monoid of nonzero \emph{monomial ideals} of $R$ (i.e., ideals generated by monomials). 

The study of the arithmetic properties of $\mathcal{P}_{{\rm fin}}(\mathbb{N})$ was started by Fan and Tringali in \cite{Fan-Tr18}. In particular, it follows from \cite[Prop.s 3.2 and 3.5]{Fan-Tr18} that $\mathcal{P}_{{\rm fin}}(\mathbb{N})$ is a commutative, reduced, unit-cancellative semigroup with $\{0\}$ as the identity element. Among their results, Fan and Tringali show that for any $i\in\mathbb{N}^+$, the sets $\{0,i\}$ and$\{ 0\}\cup \{2j+1\colon j\in\llb 0,i\rrb\}$ are atoms, while $\{1\}\cup\{2j\colon j\in\llb 0,i\rrb\}$ is an atom for $i\in\mathbb{N}_{\ge 2}$ (see \cite[Prop.s 4.1(iv) and 4.2]{Fan-Tr18}). The correspondence $\Phi$ in \eqref{eq: map} sends these sets to the ideals $\mathfrak{b}_i(X,Y)$, $\mathfrak{c}_{2i+1}(X,Y)$, and $\mathfrak{c}_{2i}(X,Y)$, respectively. In \cite[Prop.~5.10]{Ge-Kh21}, it is proved that $\mathfrak{b}_i(X,Y)$ and $\mathfrak{c}_{2i+1}(X,Y)$ are atoms of $\I(R)$ for any $i\in\mathbb{N}^+$, while $\mathfrak{c}_{2i}(X,Y)$ is an atom only for $i\ge 3$. 
Note that, \emph{a priori}, the above correspondence ensures only that these ideals are atoms of the submonoid $\mathcal{M}_2(X,Y)$, as monoid isomorphisms are clearly atom-preserving. Indeed, it turns out that in general $\Phi$ does {\it not} send atoms of $\mathcal{P}_{{\rm fin}}(\mathbb{N})$ to atoms of $\I(R)$, as the set $C=\{0,1,2,4\}$ is easily seen to be an atom of $\mathcal{P}_{{\rm fin}}(\mathbb{N})$, while, when $2$ is a unit of $D$, (see \cite[Remark 5.11]{Ge-Kh21})
\[I_C=\mathfrak{c}_4(X,Y)=\langle X^4,X^3Y,X^2Y^2,Y^4\rangle=\langle X^2,XY+Y^2\rangle\langle X^2,XY-Y^2\rangle.\]
We will show in Proposition \ref{prop: Phi atoms to atoms} that $\Phi$ sends every atom of $\mathcal{P}_{{\rm fin}}(\mathbb{N})$ to an atom of $\Mon(R)$. In particular, $I_C$ is an atom of $\Mon(R)$.

\section{Atoms and sets of lengths in $\I(R)$}\label{sec: atoms in I(R)}

Throughout this section, we maintain the notation introduced in Section \ref{subsec: polynomial rings and ideals}, and assume that $\I(R)$ is a $\mathsf{BF}$-monoid. This assumption is satisfied if, for instance, $D$ is noetherian (in particular, if $D=K$). Our aim is to find new classes of atoms in $\I(R)$, by extending the methods introduced in \cite[Sect. 5]{Ge-Kh21} to further families of ideals of $R$, corresponding (via the map $\Phi$ in \eqref{eq: map} from $\mathcal{P}_{\rm fin}(\mathbb{N})$ to $\I(R)$) to atoms of the \emph{reduced power monoid} $\mathcal{P}_{{\rm fin},0}(\mathbb{N})$ of $\mathbb{N}$, i.e., the monoid of nonempty finite subsets of $\mathbb{N}$ containing $0$ with setwise addition. Note that $\mathscr A(\mathcal P_{\rm fin}(\mathbb N)) = \mathscr A(\mathcal P_{{\rm fin},0}(\mathbb N))\cup\{\{1\}\}$, since every finite set of positive integers $A=\{a_1,\dots, a_n\}$ with $n\ge 2$ or $a_1\ge 2$ can be written as $A=(A-1)+\{1\}$. It is also worth noting that the monoid homomorphism $\Phi$ restricts to a monoid isomorphism between $\mathcal{P}_{{\rm fin},0}(\mathbb{N})$ and the ideals of $R$ generated, for some $m\in\mathbb{N}$, by $X^m$, $Y^m$ and a set of monomials in $X$ and $Y$ of degree $m$ (cf. \cite[Example 5.13]{Ge-Kh21}).

The following lemma, generalizing the arguments used in \cite{Ge-Kh21}, collects a number of necessary conditions for the existence of a nontrivial factorization for an ideal $I_A$, where $A\in\mathcal{P}_{{\rm fin},0}(\mathbb{N})$. They will be a key tool in the proofs of the main results of this section, namely Theorem \ref{thm: sum-free} and Theorem \ref{thm: I_B atom}.

\begin{lemma}\label{lem: standard}
Let $m\in\mathbb{N}^+$, $\mathcal{G}$ be a set of monomials of degree $m$ in $X$ and $Y$, and let $I=\langle X^m,Y^m,\mathcal{G}\rangle$ be an ideal of $R$. Suppose that $I=\mathfrak{a}\cdot\mathfrak{b}$, for some $\mathfrak{a},\mathfrak{b}\in\I(R)$ such that $I\subsetneq \mathfrak{a}$, $\mathfrak{b}\subsetneq R$, $\mdeg(\mathfrak{a})=d$, $\mdeg(\mathfrak{b})=e$, ${\rm dim}_K(\mathfrak{a}_K[d])=r$ and ${\rm dim}_K(\mathfrak{b}_K[e])=s$. Then, the following hold.

\vspace{.05cm}

\begin{enumerate*}[label=\textup{(\roman{*})}, mode=unboxed]
    \item \label{lem: standard1} $m=d+e$ and $I_K[m]=\mathfrak{a}_K[d]\cdot\mathfrak{b}_K[e]$.
\end{enumerate*}

\vspace{.05cm}

\begin{enumerate*}[label=\textup{(\roman{*})}, resume, mode=unboxed]
    \item \label{lem: standard2} $d\ge 1, e\ge 1$.
\end{enumerate*}

\vspace{.05cm}

\begin{enumerate*}[label=\textup{(\roman{*})}, resume, mode=unboxed]
    \item \label{lem: standard3} $\mathfrak{a}_K[d]$ has dimension $r\ge 2$ and admits a $K$-basis $\{f_1,\dots,f_r\}$ with $X^d=\ini(f_1)>\ini(f_2)>\dots>\ini(f_r)$. Analogously, $\mathfrak{b}_K[e]$ has dimension $s\ge 2$ and admits a $K$-basis $\{g_1,\dots,g_s\}$ with $X^e=\ini(g_1)>\ini(g_2)>\dots>\ini(g_s)$.
\end{enumerate*}

\vspace{.05cm}

\begin{enumerate*}[label=\textup{(\roman{*})}, resume, mode=unboxed]
    \item \label{lem: standard4} For every $j\in\llb 2,r\rrb$ there exists $a_j\in\llb 1,d\rrb$ such that $\ini(f_j)=X^{d-a_j}Y^{a_j}$. Analogously, for each $k\in\llb 2,s\rrb$ there exists $b_k\in\llb 1,e\rrb$ such that $\ini(g_k)=X^{e-b_k}Y^{b_k}$.
\end{enumerate*}

\vspace{.05cm}

\begin{enumerate*}[label=\textup{(\roman{*})}, resume, mode=unboxed]
    \item \label{lem: standard5} Set $a_1=b_1=0$; then, for every $j\in\llb 1,r\rrb$ and every $k\in\llb 1,s\rrb$ we have
    \[X^{m-(a_j+b_k)}Y^{a_j+b_k} \in \{X^m,Y^m\}\cup\mathcal{G}.\]
\end{enumerate*}
\end{lemma}
\begin{proof} It follows from Remark \ref{rem: min-deg of fg ideal} that $\mdeg(I)=m$, and item \ref{lem: standard1} then follows from equation \eqref{eq: additivity min-deg} and Lemma \ref{lem_5.6}\ref{lem_5.6(i)}; item \ref{lem: standard2} follows from \cite[Lemma 5.8]{Ge-Kh21}, while item \ref{lem: standard4} is a consequence of \ref{lem: standard3} and \ref{rem: XY}.

\medskip

Since ${\rm dim}_K(\mathfrak{a}_K[d])=r$ and ${\rm dim}_K(\mathfrak{b}_K[e])=s$, we get from \cite[Lemma 5.9]{Ge-Kh21} that there are bases $\{f_1,\dots,f_r\}$ of $\mathfrak{a}_K[d]$ and $\{g_1,\dots,g_s\}$ of $\mathfrak{b}_K[e]$ with $\ini(f_1)>\ini(f_2)>\dots>\ini(f_r)$ and $\ini(g_1)>\ini(g_2)>\dots>\ini(g_s)$. Note that $\ini(f_j)$ and $\ini(g_k)$ are monomials of degree $d$ and $e$ respectively, for all $j\in\llb 1,r\rrb$ and $k\in\llb 1,s\rrb$. If $r=1$, we get from \ref{lem: standard1} that
    \[I_K[m]\ni X^m = \sum_{k=1}^s \alpha_kf_1g_k\ \quad {\rm for\ some}\ \alpha_1,\dots,\alpha_s\in K.\]
Thus, we obtain $f_1=X^d$, implying in turn that $Y^m\notin \spa_K\{f_1\}\cdot\mathfrak{b}_K[e]$, contradicting \ref{lem: standard1}. Therefore, it must be $r\ge 2$, and similarly we obtain that also $s\ge 2$. On the other hand, if $\ini(f_1)\ne X^d$, then for every $f\in\mathfrak{a}_K[d]\setminus\{0\}$ we have $\ini(f)<X^d$, hence $\ini(fg)<X^m$ for every  $f\in\mathfrak{a}_K[d]\setminus\{0\}$,  $g\in\mathfrak{b}_K[e]\setminus\{0\}$, and so $X^m\notin I_K[m]$, contradicting \ref{lem: standard1}. Thus, $\ini(f_1)=X^d$, and analogously we can prove that also $\ini(g_1)=X^e$. This concludes the proof of item \ref{lem: standard3}.

\medskip

To prove item \ref{lem: standard5}, consider, for any $j\in\llb 1,r\rrb$ and $k\in\llb 1,s\rrb$,
\[\ini(f_jg_k) = \ini(f_j)\ini(g_k) = X^{d-a_j+e-b_k}Y^{a_j+b_k} = X^{m-(a_j+b_k)}Y^{a_j+b_k},\]
where we have used \ref{lem: standard4} and \ref{lem: standard1}. Since $f_jg_k\in I_K[m] = \spa_K(\{X^m,Y^m\}\cup\mathcal{G})$, we must also have $\ini(f_jg_k)\in I_K[m]$, and the conclusion follows from the above equalities.
\end{proof}

\subsection{Atoms constructed from sum-free sets of positive integers}\label{subsec: sum-free} We are now in the position to extend the results on the atomicity of $\mathfrak{b}_i(X,Y)$ and $\mathfrak{c}_{2i+1}(X, Y)$ to a larger family of ideals. Before proceeding, let us recall a classical definition from additive combinatorics: a nonempty subset $A\subseteq \mathbb{N}$ is \evid{sum-free} if $A\cap(A+A)=\emptyset$. Easy examples of sum-free sets include any set of odd positive integers and any subset of $\{\lceil{(n+1)/2}\rceil,\dots,n\}$, where $n\in\mathbb N$ and $\lceil\ \rceil$ denotes the usual ceiling function. In fact, Green proved in \cite{Green04} that these examples already give the correct order of magnitude for the number of sum-free sets: more precisely, there are $O(2^{\nicefrac{n}{2}})$ sum-free subsets of $\{1,\dots,n\}$ for $n\to\infty$, as predicted by the Cameron-Erd\H{o}s Conjecture.

Clearly, a finite sum-free set is contained in $\mathcal{P}_{{\rm fin}}(\mathbb{N})\setminus\mathcal{P}_{{\rm fin},0}(\mathbb{N})$. It follows that no sum-free set different from $\{1\}$ is an atom of $\mathcal{P}_{{\rm fin}}(\mathbb{N})$ (see the remarks at the beginning of the section). In preparation for Theorem \ref{thm: sum-free}, we note that adjoining the zero element to sum-free subsets of $\mathbb{N}$ turns them into atoms of $\mathcal{P}_{{\rm fin}}(\mathbb{N})$.

\begin{lemma}\label{lem: sum-free with zero is atom}
If $A$ is finite and sum-free, then $\{0\}\cup A$ is an atom of $\mathcal{P}_{\rm fin}(\mathbb{N})$.
\end{lemma}
\begin{proof}
Suppose that there exist $B,C\in\mathcal{P}_{\rm fin}(\mathbb{N})\setminus\{\{0\}\}$ such that
\[\{0\}\cup A = B+C.\]
Firstly, note that $0\in B\cap C$, for otherwise $0$ cannot be an element of $B+C$. Furthermore, since $B$ and $C$ both differ from $\{0\}$, both $\beta:=\max{B}$ and $\gamma:=\max{C}$ are positive integers. The set $B+C$ contains, in particular, the elements $\beta,\gamma$ and $\beta+\gamma$; since they are all positive, they necessarily belong to $A$, but this is a contradiction, as $A$ was supposed to be sum-free.
\end{proof}

For every $i\in\mathbb{N}^+$, the sets $\beta_i=\{i\}$ and $\delta_{2i+1}=\{2\llb 0,i\rrb +1\}$ are clearly sum-free. We then get from Lemma \ref{lem: sum-free with zero is atom} an alternative proof of the fact (mentioned in the introduction and proved in \cite{Fan-Tr18}) that $\{0\}\cup \beta_i$ and $\{0\}\cup \delta_{2i+1}$ are atoms of $\mathcal{P}_{\rm fin}(\mathbb{N})$. The next theorem (the first of the main results of this section) provides a general method for constructing atoms in $\I(R)$ from sum-free subsets of $\mathbb{N}$ and, as a special case, yields the atomicity of $\mathfrak{b}_i(X,Y)=I_{\{0\}\cup \beta_i}$ and $\mathfrak{c}_{2i+1}(X, Y)=I_{\{0\}\cup \delta_{2i+1}}$.

\begin{theorem}\label{thm: sum-free}
    If $A$ is a finite sum-free subset of $\mathbb{N}$, then $I_{\{0\}\cup A}$ is an atom of $\I(R)$.
\end{theorem}
\begin{proof}
    Suppose that $I_{\{0\}\cup A}=\mathfrak{a}\cdot\mathfrak{b}$ for some $\mathfrak{a},\mathfrak{b}\in\I(R)\setminus\{R\}$. Set $m=\max{A}=\mdeg(I_{\{0\}\cup A})$, $d=\mdeg(\mathfrak{a})$ and $e=\mdeg(\mathfrak{b})$. Since $\{0\}\cup A\in\mathcal{P}_{{\rm fin},0}(\mathbb{N})$, we can apply Lemma \ref{lem: standard}. Therefore, we obtain $d\ge 1$, $e\ge 1$, $m=d+e$,
    \[{(I_{\{0\}\cup A})}_K[m]=\mathfrak{a}_K[d]\cdot\mathfrak{b}_K[e],\]
    $\mathfrak{a}_K[d]=\spa_K\{f_1,\dots,f_r\}$, where $r\ge 2$ and $X^d=\ini(f_1)>\dots>\ini(f_r)$, $\mathfrak{b}_K[e]=\spa_K\{g_1,\dots,g_s\}$, where $s\ge 2$ and $X^e=\ini(g_1)>\dots>\ini(g_s)$. Let $\ini(f_2)=X^{d-a}Y^a$ for some $a\in\llb 1,d\rrb$ and $\ini(g_2)=X^{e-b}Y^b$ for some $b\in\llb 1,e\rrb$. Then Lemma \ref{lem: standard}\ref{lem: standard5} yields
    \[X^{m-a}Y^a, X^{m-b}Y^b, X^{m-(a+b)}Y^{a+b} \in \{X^{m-c}Y^c\colon c\in\{0\}\cup A\},\]
    i.e., $a,b,a+b\in\{0\}\cup A$. Since both $a$ and $b$ are positive, we actually get $a,b,a+b\in A$, but this contradicts the assumption that $A$ is a sum-free set.
\end{proof}
In light of the remarks at the beginning of this section, the above theorem gives a substantial number of examples of atoms of $\I(R)$. As a further corollary, we also obtain an $\I(R)$-analogue of \cite[Prop.~4.1(v)]{Fan-Tr18}.
\begin{corollary}
Let $m\in\mathbb{N}_{\ge 2}$. For any integer $j\in\llb 1,m-1\rrb$, the ideal $I=I_{\{0,j,m\}}=\langle X^m, X^{m-j}Y^j, Y^m\rangle$ is an atom of $\I(R)$ if and only if $m\ne 2j$.
\end{corollary}
\begin{proof}
If $m=2j$, then $I={\langle X^j,Y^j\rangle}^2$ is clearly not an atom. Conversely, if $m\ne 2j$, then the set $\{j,m\}$ is sum-free, and $I=I_{\{0,j,m\}}$ is an atom by Theorem \ref{thm: sum-free}.
\end{proof}

\subsection{A further family of atoms}\label{subsect: I_B}
For every $i\in \mathbb N_{\ge 2}$, the set $\delta_{2i}=\{1, 2\llb 1, i\rrb\}$ is clearly not sum-free. Nonetheless, as recalled in Section \ref{subsec: power monoids}, $\mathfrak{c}_{2i}(X,Y)=I_{\{0\}\cup \delta_{2i}}$ is an atom for every $i\ge 3$. In this section, we will provide a further class of atoms not covered by Theorem \ref{thm: sum-free}.

Following \cite[Prop.~4.10]{Fan-Tr18}, fix a positive integer $n\ge2$ and let $a_1,\dots,a_n,a_{n+1}$ be positive integers satisfying the following conditions: 
\begin{enumerate}[label=\textup{(C\arabic{*})}, mode=unboxed]
    \item\label{cond_1} $a_{n+1}=a_1+\dots+a_{n-1}+2a_n$,
    \item\label{cond_2} $a_{i+1}>2(a_1+\dots+a_i)$ for all $i\in\llb 1,n-1\rrb$.
\end{enumerate}
Further, define
\[A_n:=\left\{a_I \colon I\subseteq \llb 1,n-1\rrb\right\},\]
\[B_n:= A_n \cup \{a_{\llb 1, n \rrb}\} \cup (A_n +a_{n+1}),\]
\[C_n:=\left\{a_I \colon I\subseteq\llb 1,n+1\rrb\right\},\]
where we set $a_\emptyset = 0$ and $a_I = \sum_{i\in I}a_i$ for every nonempty subset $I \subseteq \llb 1,n+1\rrb$.

In \cite[Prop.~4.10(i)]{Fan-Tr18}, it is proved that, for every $n\ge 2$, the set $B_n$ is an atom of $\mathcal P_{{\rm fin},0}(\mathbb N)$.

\begin{remark}\label{rem:m/2}
By the definition of the $a_i$'s it is clear that $\max(B_n)=a_{\llb 1,n-1\rrb}+a_{n+1}=a_{\llb 1,n-1\rrb\cup\{{n+1}\}}$. By condition \ref{cond_1}, if $a_{\llb 1,n-1\rrb}$ is odd (resp. even), then $a_{n+1}$ is odd (resp. even), hence $m=\max(B_n)$ is always even. In particular, $m=a_{\llb 1,n-1\rrb}+a_{n+1}=2 a_{\llb 1,n-1\rrb}+2a_{n}=2a_{\llb 1,n\rrb}$. 
\end{remark}

Let us focus on the ideals of $R$
\[I_{B_n}=\langle X^{m-b}Y^b \colon b\in B_n\rangle\text{ and }I_{C_n}=\langle X^{a_{\llb1,n+1\rrb}-c}Y^c \colon c\in C_n\rangle. \]

\begin{remark}\label{rem: I_C as product of b_i's}
The following identities can be easily verified by a direct computation or using \cite[Prop.~4.10 (ii)]{Fan-Tr18} and the monoid homomorphism $\Phi\colon\ \mathcal{P}_{{\rm fin}}(\mathbb{N})\to\I(R)$.
\begin{equation}\label{eq: I_C_n decompositions}
I_{C_n}=\prod_{i=1}^{n+1}\mathfrak b_{a_i}=\mathfrak{b}_{a_n}\cdot I_{B_n}.    
\end{equation}
We recall here that $\mathfrak{b}_i = \langle X^i, Y^i \rangle \in \mathscr A(\I(R))$ for every $i\in \mathbb N^+$.
\end{remark}

The next theorem provides a new class of atoms of $\I(R)$, namely the ideals $I_{B_n}$, and gives bounds on the set of lengths of the ideals $I_{C_n}$.

\begin{theorem} \label{thm: I_B atom}

\begin{enumerate*}[label=\textup{(\arabic{*})}, resume, mode=unboxed]
\item\label{thm: I_B atom(1)} $I_{B_n}$ is an atom of $\I(R)$ for every $n \geq 2$.
\end{enumerate*}

\vspace{.2cm}

\begin{enumerate*}[label=\textup{(\arabic{*})}, resume, mode=unboxed]
\item\label{thm: I_B atom(2)} $\{2, n+1\} \subseteq \mathsf{L}_{\I(R)}(I_{C_n}) \subseteq \llb 2, n+1 \rrb$.
\end{enumerate*}
\end{theorem} 
 
Before proving this result, we need some technical lemmas based on the properties of the positive integers $a_1, \dots, a_{n+1}$. 

We first observe that condition \ref{cond_2} implies that, for every $k\in \llb 1, n\rrb$, the map
\[
\{0,1,2\}^k\to \mathbb{N}, \quad (\varepsilon_1,\dots,\varepsilon_k)\mapsto\sum_{i=1}^k\varepsilon_i a_i
\]
is injective. Indeed, suppose that two distinct coefficient vectors $(\varepsilon_i)$ and $(\eta_i)$ have the same image, and let $t$ be the largest index for which $\varepsilon_t\ne\eta_t$. Then
\[(\varepsilon_t-\eta_t)a_t=-\sum_{i=1}^{t-1}(\varepsilon_i-\eta_i)a_i.\]
The absolute value of the left-hand side is at least $a_t$, while that of the right-hand side is at most $2 a_{\llb 1,t-1\rrb}$, but this is impossible as it contradicts \ref{cond_2}.

\begin{lemma}\label{lemma_1}
Let $n,k$ be positive integers such that $n\ge 2$ and $k\le n$ and let $I,J\subseteq \llb 1,k\rrb$. Then $a_I+a_J=a_H$ for some $H\subseteq \llb 1,k\rrb$ if and only if $I\cap J=\emptyset$.
\end{lemma}
\begin{proof}
    If $I\cap J=\emptyset$, then $a_I+a_J=a_{I\cup J}$. Conversely, suppose that $a_I+a_J=a_H$. By the injectivity observation above, the coefficient of each $a_i$ on both sides must agree. Since the coefficient on the right belongs to $\{0,1\}$, no index can belong to both $I$ and $J$. Hence $I\cap J=\emptyset$ (and necessarily $H=I\cup J$).
\end{proof}

\begin{lemma}\label{lemma_2}
For a positive integer $n\ge 2$ and $I,J,H\subseteq \llb 1,n-1\rrb$, assume  that
\[a_{\llb 1,n-1\rrb}+a_H=a_{I}+a_J=a_{I\cup J}+a_{I\cap J}.\]
Then $I\cup J=\llb 1,n-1\rrb$ and $H=I\cap J$.
\end{lemma}
\begin{proof}
    By the injectivity observation above, the equality $a_{\llb 1,n-1\rrb}+a_H=a_{I\cup J}+a_{I\cap J}$ implies that the coefficients of each $a_i$ on the two sides agree. On the left, every $a_i$ has coefficient at least $1$, and therefore $I\cup J=\llb1,n-1\rrb$. Moreover, the coefficient of $a_i$ is $2$ on the left if and only if $i\in H$, and on the right if and only if $i\in I\cap J$. Hence $H=I\cap J$.
\end{proof}

\begin{lemma}\label{lemma_3}
Let $n\ge 2$. The following statements hold:

\vspace{.2cm}
\begin{enumerate*}[label=\textup{(\roman{*})}, mode=unboxed]
\item\label{lemma_3(1)} For every $I, J\subseteq \llb 1, n-1 \rrb$, $a_I + a_J \in A_n$ if and only if $I\cap J = \emptyset$.
\end{enumerate*}

\vspace{.2cm}
\begin{enumerate*}[label=\textup{(\roman{*})}, resume, mode=unboxed]
\item\label{lemma_3(2)} For every $I\in\mathcal{P}(\llb 1,n-1\rrb)$ and $J\in \mathcal{P}( \llb 1, n-1 \rrb) \cup \{ Y \cup \{n+1\} \colon Y\in \mathcal{P}( \llb 1, n-1 \rrb)\}$ (note that $\cup$ is a disjoint union), $a_I + a_J \in B_n$ if and only if $I\cap J = \emptyset$. Moreover, $a_{\llb 1, n \rrb} + a_I \in B_n$ if and only if $I = \emptyset$.
\end{enumerate*}

\vspace{.2cm}
\begin{enumerate*}[label=\textup{(\roman{*})}, resume, mode=unboxed]
\item\label{lemma_3(3)} For every $I, J\subseteq \llb 1, n+1 \rrb$, $a_I + a_J \in C_n$ if and only if either $I\cap J = \emptyset$ or $I\cup J=\llb 1,n\rrb$ and $n\in I\cap J$.
\end{enumerate*}
\end{lemma}

\begin{proof}
\ref{lemma_3(1)} The statement is an immediate consequence of Lemma \ref{lemma_1} with $k=n-1$, since $a_I+a_J\in A_n$ if and only if $a_{I}+a_{J}=a_{H}$ for some $H\subseteq\llb 1,n-1\rrb$.
  
  \vspace{.2cm}  
\ref{lemma_3(2)} We first assume $I\cap J=\emptyset$. If $J\in\mathcal{P}( \llb 1, n-1 \rrb)$, then $a_I+a_J\in A_n\subseteq B_n$ by \ref{lemma_3(1)}. If $J=X\cup\{n+1\}$ for some $X\in \mathcal{P}(\llb 1, n-1\rrb)$, then $I\cap X=\emptyset$ and, by \ref{lemma_3(1)}, $a_I+a_X\in A_n$. Thus $a_I+a_J=a_I+a_X+a_{n+1}\in A_n+a_{n+1}\subseteq B_n$.
    
    On the other hand, let $I\subseteq\llb 1,n-1\rrb$, $J\in \mathcal{P}( \llb 1, n-1 \rrb) \cup \{ Y \cup \{n+1\} \colon Y\in \mathcal{P}( \llb 1, n-1 \rrb)\}$ and assume $a_I+a_J\in B_n$. Then $a_I+a_J=a_{H}$, where either $H\subseteq\llb 1, n-1\rrb$, $H=\llb1,n\rrb$ or $H=X\cup\{n+1\}$ with $X\in\mathcal{P}(\llb 1,n-1\rrb)$. Let us first observe that $H\ne\llb1,n\rrb$. Indeed, if we assume that $H = \llb1,n\rrb$, if $n+1 \in J$, then $a_I + a_J \geq a_{n+1} > a_H$, while if $n+1 \notin J$ (so $I, J \subseteq \llb1,n-1\rrb$), $a_I + a_J \leq 2a_{\llb1,n-1\rrb} < a_n \leq a_H$.
    
    If $H\in \mathcal{P}(\llb 1, n-1\rrb)$, then $J\in \mathcal{P}(\llb 1, n-1\rrb)$ and $a_H\in A_n$, hence $I\cap J=\emptyset$, by \ref{lemma_3(1)}. If $H=X\cup\{n+1\}$ for some $X\in \mathcal{P}(\llb 1, n-1\rrb)$, then $J=Y\cup\{n+1\}$ for some $Y\in \mathcal{P}(\llb 1, n-1\rrb)$ and $a_I+a_Y=a_X$. Note that $a_X\in A_n$, so, by \ref{lemma_3(1)}, $I\cap Y=\emptyset$ and hence $I\cap J=\emptyset$.
    
    Finally, we prove that $a_{\llb 1,n\rrb}+a_I\in B_n$ only if $I=\emptyset$ (the reverse implication is trivial). Assume $a_{\llb 1,n\rrb}+a_I=a_H$, where either $H\subseteq\llb 1, n-1\rrb$, $H=\llb1,n\rrb$ or $H=X\cup\{n+1\}$ with $X\subseteq\llb 1, n-1\rrb$. First, observe that $H$ cannot contain $n+1$, otherwise $a_H\ge a_{n+1}=a_{\llb 1,n\rrb}+a_n> a_{\llb 1,n\rrb}+2a_{\llb 1, n-1\rrb}\ge a_{\llb 1,n\rrb}+a_I$, so we are left to check only the first two of the above conditions. If $H\subseteq\llb 1, n-1\rrb$, Lemma \ref{lemma_1} implies $\llb1,n\rrb\cap I=\emptyset$ and this can happen only if $I=\emptyset$. If $H=\llb1,n\rrb$, then $a_I=0$ and $I = \emptyset$.
   
    \vspace{.2cm} 
\ref{lemma_3(3)} First, if $I\cap J=\emptyset$, then $a_I+a_J=a_{I\cup J}\in C_n$, while if $I\cup J=\llb 1, n\rrb$ and $n\in I\cap J$, then $a_I+a_J=a_{I\cup J}+a_{I\cap J}=a_{\llb 1,n\rrb}+a_n+a_{((I\cap J)\setminus \{n\})}=a_{((I\cap J)\setminus \{n\})\cup\{n+1\}}\in C_n$.
 
    On the other hand, assume $a_I+a_J\in C_n$, i.e., there exists $H\subseteq\llb 1,n+1\rrb$ such that $a_H=a_I+a_J$. We consider two cases, according to whether $n\in I\cap J$. If $n\in I\cap J$, then $n+1\notin I\cup J$, as otherwise
   \[ a_H=a_I+a_J\ge 2a_n+a_{n+1}>a_{\llb 1,n-1\rrb}+a_{\llb 1,n+1\rrb}>a_{\llb 1,n+1\rrb}=\max(C_n). \]
    Now consider the equality
    \[ a_H = a_{I\cup J}+a_{I\cap J} = a_{(I\cup J)\setminus\{n\}}+a_{(I\cap J)\setminus\{n\}}+2a_n; \]
    by the injectivity observation above, we must have $n+1\in H$, as the coefficient of $a_n$ on the left-hand side is at most $1$, while it equals $2$ on the right-hand side. The above equality can therefore be rewritten (using \ref{cond_1}) as
    \[ a_{H\setminus\{n+1\}} + a_{\llb 1,n-1\rrb} = a_{(I\cup J)\setminus\{n\}}+a_{(I\cap J)\setminus\{n\}}. \]
    The injectivity observation above now implies that $n\not\in H\setminus\{n+1\}$. Therefore, Lemma \ref{lemma_2} yields $(I\cup J)\setminus\{n\}=\llb 1,n-1\rrb$, i.e., $I\cup J=\llb 1,n\rrb$.

    It remains to consider the case $n\notin I\cap J$. If $n+1\notin I\cup J$, then $a_H=a_I+a_J\le a_n+2a_{\llb 1,n-1\rrb} < 2a_n < a_{n+1}$, so $n+1\notin H$. We can then apply Lemma \ref{lemma_1} to conclude that $I\cap J=\emptyset$. If $n+1\in I\cup J$, note that $n+1\notin I\cap J$, as $a_H\le a_{\llb 1,n+1\rrb}<2a_{n+1}$. Moreover, $n+1\in H$, otherwise $a_{H}\le a_{\llb 1,n\rrb}<a_{n+1}$. We can then rewrite $a_H=a_I+a_J$ as $a_{H\setminus\{n+1\}} = a_{I\setminus\{n+1\}}+a_{J\setminus\{n+1\}}$, and conclude $I\cap J=\emptyset$ by Lemma \ref{lemma_1}.
\end{proof}

We are finally ready to prove Theorem \ref{thm: I_B atom}.

\begin{proof}[Proof of Theorem \ref{thm: I_B atom}] \ref{thm: I_B atom(1)} As $\I(R)$ is unit-cancellative and reduced, to prove that $I_{B_n}$ is an atom of $\I(R)$ it suffices to show that there cannot be $\mathfrak{a},\mathfrak{b}\in \I(R)$ such that $I_{B_n} \subsetneq \mathfrak{a}, \mathfrak{b} \subsetneq R$, and $I_{B_n} = \mathfrak{a} \cdot \mathfrak{b}$.
Let us assume to the contrary that such $\mathfrak a$ and $\mathfrak b$ exist.
Set $\mdeg(\mathfrak{a}) = d$ and $\mdeg(\mathfrak{b}) = e$. Then, by Lemma \ref{lem: standard}, we have $d\geq 1$, $e \geq 1$, $d+e = m=\max(B_n)$,
	\begin{equation}\label{eq:atom1}
	(I_{B_n})_K[m]=\spa_K\{X^{m-b}Y^b \colon b\in B_n\} = \mathfrak{a}_K[d] \cdot \mathfrak{b}_K[e],
	\end{equation}
Furthermore, there exist $r\ge 2$ linearly independent elements $f_1, \dots, f_r\in \mathfrak{a}_K[d]$ such that $X^d=\ini(f_1) > \dots > \ini(f_r)$ and $\mathfrak{a}_K[d] = \spa_K\{f_1,\ldots, f_r\}$, and $s\ge 2$ linearly independent elements $g_1, \dots, g_s\in \mathfrak{b}_K[e]$ such that $X^e=\ini(g_1) > \dots > \ini(g_s)$ and $\mathfrak{b}_K[e] = \spa_K\{g_1,\ldots, g_s\}$, whence for each $i\in\llb 2,r\rrb$ there is some $\alpha_i\in\llb 1,d\rrb$ such that $\ini(f_{i}) = X^{d-\alpha_i}Y^{\alpha_i}$ and for each $j\in\llb 2,s\rrb$ there is some $\beta_j\in\llb 1,e\rrb$ such that $\ini(g_{j}) = X^{e-\beta_j}Y^{\beta_j}$.
By symmetry, we can assume without loss of generality that $d \leq m / 2 = a_{\llb 1, n \rrb}$ (see Remark \ref{rem:m/2}). Moreover, $d\ge 1$ implies $e<m=a_{\llb 1,n-1\rrb\cup\{n+1\}}$.

\medskip

Set $\mathcal{P} := \mathcal{P}(\llb 1,n-1\rrb)$, $\mathcal{P}^+ := \mathcal{P}\setminus\{\emptyset\}$, $\mathcal{Q}:= \mathcal{P} \cup \{ Y \cup \{n+1\} \colon Y\in \mathcal{P}\}$ and $\mathcal{Q}^+ := \mathcal{Q}\setminus\{\emptyset\}$. For every fixed $i\in\llb 2,r\rrb $, $f_ig_1$ belongs to $(I_{B_n})_K[m]$, whence also $\ini(f_ig_1)=X^{m-\alpha_i}Y^{\alpha_i}\in(I_{B_n})_K[m]$. Since $1\leq \alpha_i\leq d\leq a_{\llb 1,n\rrb}$, by \eqref{eq:atom1} it follows that $\alpha_i\in B_n\cap \llb 1, a_{\llb 1, n \rrb}\rrb=A_n\cup \{a_{\llb 1, n \rrb}\}\setminus{\{0\}}$, so we can write $\alpha_i = a_{I_i}$ for some $I_i\in \mathcal{P}^+\cup \{ \llb1,n\rrb \}$. Similarly, by looking at $\ini(f_1g_j)$ with $j\in \llb 2,s\rrb $, we conclude that $\beta_j\in B_n\setminus{\{0, m\}}$, so we can write $\beta_j = a_{J_j}$ for some $J_j\in \mathcal{Q}^+\cup \{ \llb1,n\rrb \}\setminus{\{\llb1,n-1\rrb\cup \{n+1\}\}}$. Thus, $\ini(f_ig_j) = X^{m-(a_{I_i}+a_{J_j})}Y^{a_{I_i}+a_{J_j}}$ and by \eqref{eq:atom1} we conclude $a_{I_i}+a_{J_j} \in B_n\setminus{\{0\}}$. We split the rest of the argument in three cases, according to the structure of the sets $I_i$ and $J_j$ for $i\in\llb 2,r\rrb,\ j\in\llb 2,s\rrb$.

\medskip

CASE 1: assume that for some $i\in\llb2,r\rrb$,  $\ini(f_{i}) = Y^{a_{ \llb1,n\rrb }}$, i.e. $I_i=\llb 1,n\rrb$. This can only happen if $d =e= m/2 = a_{ \llb1,n\rrb }$. For some $j\in \llb2,s\rrb$, consider $\ini(g_j)=X^{e-a_{J_j}}Y^{a_{J_j}}$, with $a_{J_j}\le e$. If $a_{J_j} < e= a_{ \llb1,n\rrb }$, then ${J_j} \in \mathcal{P}^+$ and, by Lemma \ref{lemma_3}\ref{lemma_3(2)}, $a_{J_j} + a_{ \llb1,n\rrb } \notin B_n$, implying $f_ig_j\notin (I_{B_n})_K[m]$. Hence, the only possibility is $a_{J_j} = e$, so we must have $\dim \mathfrak{b}_K[e] = 2$. An analogous argument shows that also $\dim \mathfrak{a}_K[d] = 2$, and so $\dim( \mathfrak{a}_K[d]\cdot \mathfrak{b}_K[e] )\leq 4$, contradicting \eqref{eq:atom1}. In fact $\dim(I_{B_n})_K[m]=|B_n|=2^n+1\ge5$.

\medskip

CASE 2: assume that for some $j\in\llb2,s\rrb$,  $\ini(g_j) = X^{e-a_{ \llb1,n\rrb }}Y^{a_{ \llb1,n\rrb }}$, i.e. ${J_j}=\llb 1,n\rrb$. By CASE 1, we can assume that $I_i\neq\llb 1,n\rrb$, for all $i\in\llb 2,r\rrb$ and hence ${I_i}\in \mathcal{P}^+$ for all such $i$. Then, by Lemma \ref{lemma_3}\ref{lemma_3(2)} we conclude that $a_{I_i} + a_{ \llb1,n\rrb } \notin B_n$ for every $i\in\llb 2,r\rrb$. It then follows that $\ini(f_{i}g_{j})\notin (I_{B_n})_K[m]$ for every $i\in\llb2,r\rrb$, which is a contradiction.

\medskip

CASE 3: assume that for all $i\in\llb 2,r\rrb$ and all $j\in\llb 2,s\rrb$ we have ${I_i}\in \mathcal{P}^+$ and ${J_j}\in \mathcal{Q}^+\setminus{\{\llb1,n-1\rrb\cup \{n+1\}\}}$. By Lemma~\ref{lemma_3}\ref{lemma_3(2)}, we may further assume that $I_i\cap J_j = \emptyset$ for all $i,j$. Since $\ini(f_{i_1})\ne \ini(f_{i_2})$ whenever $i_1\ne i_2\in \llb 1,r\rrb$, and $|\mathcal{P}|=2^{n-1}$, we obtain  $r=\dim \mathfrak{a}_K[d]\le 2^{n-1}$. Assume that $r \in \llb 2^k+1, 2^{k+1}\rrb$, where $k\in \llb0, n-2\rrb$. To each $f_i$, where $i\in\llb 1,r\rrb$, corresponds a set $I_i\in \mathcal{P}$, and $I_{i_1}\ne I_{i_2}$ for each $i_1\ne i_2$. Since $r\ge 2^{k}+1$, the union of all the $I_i$'s has cardinality at least equal to $k+1$, since a set of cardinality at most $k$ has no more than $2^k$ subsets. Thus, the number of sets $J_j \in \mathcal{P}$ such that $J_j\cap I_i=\emptyset$ for every $i$, is at most $2^{n-k-2}$ ($J_j$ is actually a subset of $\llb 1,n-1\rrb \setminus \bigcup I_i$). Similarly, also the number of admissible sets $J_j \in \mathcal{Q}\setminus \mathcal{P}$ is at most $2^{n-k-2}$. In total, the number of possibilities for $J_j \in \mathcal{Q}$ is at most $2^{n-k-1}$. Hence $\dim \mathfrak{b}_K[e] \leq 2^{n-k-1}$ and consequently $\dim(\mathfrak{a}_K[d]\cdot \mathfrak{b}_K[e]) \leq 2^n$, a contradiction to \eqref{eq:atom1}.

\vspace{.2cm}
\ref{thm: I_B atom(2)} $\{2, n+1\} \subseteq \mathsf{L}_{\I(R)}(I_{C_n})$ follows by Remark \ref{rem: I_C as product of b_i's}, the previous point \ref{thm: I_B atom(1)} and the fact that all ideals of the form $\mathfrak b_i$ are atoms of $\I(R)$ (by \cite[Prop.~5.10.1]{Ge-Kh21}).

For the proof of $\mathsf{L}_{\I(R)}(I_{C_n}) \subseteq \llb 2, n+1 \rrb$, we assume to the contrary that the following identity holds
\[
I_{C_n} = \prod_{i=1}^{n+2}I_i,
\]
where $I_i\in \I(R)\setminus\{R\}$ with $\mdeg(I_i) = d_i$ for every $i \in \llb 1, n+2\rrb$. We have $\sum _{i=1}^{n+2}d_i=a_{\llb 1, n+1\rrb}$ by \eqref{eq: additivity min-deg}. Thus,
\[
(I_{C_n})_K[\mu] = \spa_K\{X^{\mu-c}Y^c \colon c\in C_n\}= \prod_{i=1}^{n+2}(I_{i})_{K}[d_i],
\]
where $\mu = a_{\llb1, n+1 \rrb}=\max C_n$. As in the previous point, we may assume that \[
d_i \geq 1 \text{ and } \dim( I_{i})_K[d_i] \geq 2 \text{ for all } i \in \llb1, n+2\rrb.
\]
Moreover, let
\[
(I_{i})_K[d_i] = \spa\{f_1^{(i)},\ldots, f_{r_i}^{(i)}\},
\]
where $f_1^{(i)},\ldots, f_{r_i}^{(i)}$ are linearly independent, $\ini(f_1^{(i)}) > \ini(f_2^{(i)}) > \ldots > \ini(f_{r_i}^{(i)})$ and $\ini(f_1^{(i)}) = X^{d_i}$, for all $i \in \llb 1, n+2 \rrb$. For any $j\in\llb2,r_i\rrb$, let $\ini(f_j^{(i)}) = X^{d_i - a_{ij}}Y^{a_{ij}}$ for some $1\leq a_{ij} \leq d_i$. Then the following equation
\[
\ini\left(f_j^{(i)}\cdot \prod_{\substack{
k=1\\
k\ne i
}}^{n+2}f_1^{(k)}\right) = X^{\mu - a_{ij}} Y^{a_{ij}}
\]
shows that $a_{ij} \in C_n$ and hence $a_{ij} = a_{I_j^{(i)}}$ with $\emptyset \neq I_j^{(i)} \subseteq \llb 1, n+1 \rrb$ for every $i\in\llb1,n+2\rrb$ and $j\in \llb 2,r_i\rrb$. Similarly, $\sum_{\ell = 1}^ka_{I_{j_{\ell}}^{(i_{\ell})}} \in C_n$ for every $k\leq n+2$ and $i_s \ne i_t$ for every $s\ne t$. We claim that $n\in I_2^{(i)}$ for all $i \in \llb 1, n+2\rrb$ and that $I_2^{(i)} \subseteq \llb 1, n\rrb$ for all $i \in \llb 1, n+2\rrb$.
If $n \not \in I_2^{(i_0)}$ for some $i_0$, since $a_{I_2^{(i_0)}}+a_{I_2^{(i)}}\in C_n$ for every $i\ne i_0$, by Lemma \ref{lemma_3}\ref{lemma_3(3)} we get
\begin{equation}\label{eq: disjoint index sets}
I_2^{(i_0)} \cap I_2^{(i)} = \emptyset \text{ for all } i \ne i_0.
\end{equation}
If two of the remaining sets intersect, say $I_2^{(i)} \cap I_2^{(j)}\ne \emptyset$, Lemma \ref{lemma_3}\ref{lemma_3(3)} implies (using $a_{I_2^{(i)}} +a_{I_2^{(j)}}\in C_n$) that $I_2^{(i)} \cup I_2^{(j)}=\llb 1,n\rrb$ and $n\in I_2^{(i)} \cap I_2^{(j)}$. Then, by \eqref{eq: disjoint index sets}, we get $I_2^{(i_0)}\cap \llb 1,n\rrb=\emptyset$, forcing $I_2^{(i_0)}=\{n+1\}$  (recall that $I_2^{(i)}\ne\emptyset$ for all $i$). Hence,
\[ a_{I_2^{(i_0)}}+a_{I_2^{(i)}}+a_{I_2^{(j)}} = a_{n+1}+a_{I_2^{(i)}\cup I_2^{(j)}}+a_{I_2^{(i)}\cap I_2^{(j)}} \ge a_{n+1}+a_{\llb 1,n\rrb}+a_n > \mu, \]
contradicting the fact that $a_{I_2^{(i_0)}}+a_{I_2^{(i)}} +a_{I_2^{(j)}}\in C_n$. Thus, the sets $I_2^{(i)}$ are pairwise disjoint and so we obtain $|\cup_{i=1}^{n+2}I_2^{(i)}| \geq n+2$, contradicting the fact that $\cup_{i=1}^{n+2}I_2^{(i)} \subseteq \llb 1, n+1\rrb$. This shows that $n\in I_2^{(i)}$ for all $i$. Using Lemma \ref{lemma_3}\ref{lemma_3(3)} again, we conclude
\[
I_2^{(i)} \cup I_2^{(j)} =  \llb 1, n\rrb \text{ for all } i\ne j.
\] In particular, $n+1$ cannot be contained in any of the $I_2^{(i)}$.

Recall that $n \geq 2$, whence
\begin{align*}
\sum_{i = 1}^{n+2}a_{I_2^{(i)}} \geq \sum_{i = 1}^{4}a_{I_2^{(i)}} &= a_{I_2^{(1)}\cup I_2^{(2)}} + a_{I_2^{(1)}\cap I_2^{(2)}} + a_{I_2^{(3)}\cup I_2^{(4)}} + a_{I_2^{(3)}\cap I_2^{(4)}}\\
&= a_{ \llb 1, n\rrb} + a_{I_2^{(1)}\cap I_2^{(2)}} + a_{ \llb 1, n\rrb} + a_{I_2^{(3)}\cap I_2^{(4)}} \\
&\geq 2a_{ \llb 1, n\rrb} + 2a_n = a_{ \llb 1, n+1\rrb} + a_n > \mu,
\end{align*}
contradicting the fact that $\sum_{i = 1}^{n+2}a_{I_2^{(i)}} \in C_n$.
\end{proof}

In \cite[Prop.~4.10]{Fan-Tr18}, the authors show that $B_n$ is an atom in $\mathcal{P}_{\text{fin},0}(\mathbb{N})$ and that $\mathsf{L}_{\mathcal{P}_{\text{fin},0}(\mathbb{N})}(C_n)=\{2,n+1\}$. In light of Theorem \ref{thm: I_B atom}, it makes sense to ask whether, in analogy with the results in \cite{Fan-Tr18}, also $\mathsf{L}_{\I(R)}(I_{C_n})=\{2,n+1\}$; we note that this would immediately imply $\Delta(\I(R))=\mathbb N^+$. The question remains open, but in the next section we will prove that in the case $D=K$ the answer is negative, if we consider the analogous problem in the submonoid $\Mon(R)\subseteq\I(R)$ of nonzero monomial ideals in $R$.

\section{Arithmetical properties of $\Mon(R)$}\label{sec: atoms of Mon(R)}
In this section we focus on the arithmetical properties of the monoid $\Mon(R)$, submonoid of $\I(R)$ consisting of nonzero monomial ideals in $R=K[X_1,\dots,X_N]$, with $K$ an arbitrary field and $N\ge 2$.

We start by recalling a basic fact on monomial ideals which we will repeatedly use in the following discussion. We refer to \cite[Cor.~1.8]{En-Her} for a proof.

\begin{lemma}\label{lem: mon_gen}
Let $I\in\Mon(R)$, and $\mathcal G$ a set of monomials in $I$. Then $\mathcal G$ generates $I$ if and only if for each monomial $v\in I$ there exists $u\in\mathcal G$ such that $u\mid v$.
\end{lemma}

\begin{remark}\label{rem: mon min gen}
If $\mathcal G$ is a set of monomials, denote by $\Min\mathcal G$ the subset consisting of those monomials that are minimal with respect to divisibility among the elements of $\mathcal G$.  An important consequence of Lemma \ref{lem: mon_gen} is that every monomial ideal $I$ admits a unique generating set of monomials that is minimal with respect to inclusion \cite[Prop.~1.11]{En-Her}. 
More precisely, if $\mathcal G$ is any set of monomials such that $I=\langle\mathcal G\rangle$, then $\Min\mathcal G$ is the minimal generating set of $I$; we denote it by $G(I)$. As a consequence of Dickson's Lemma \cite[Theorem 1.9]{En-Her}, $G(I)$ is finite for every $I\in\Mon(R)$.
\end{remark} 

The following corollary of Lemma \ref{lem: mon_gen} will be used throughout the section: in essence, it tells us that when an ideal generated only by monomials in $X$ and $Y$ factors into a product of monomial ideals, the factors can also be taken to have only monomials in $X$ and $Y$ as generators. In preparation for its proof, we isolate the following special case of Remark \ref{rem: mon min gen}: if $\mathfrak a,\mathfrak b\in\Mon{(R)}$, then $G(\mathfrak a\cdot\mathfrak b) = {\Min\{fg\colon f\in G(\mathfrak a),\ g\in G(\mathfrak b)\}}$.

\begin{corollary}\label{cor: wlog mon factors}
    Suppose $I=\mathfrak a\cdot\mathfrak b$ for some $I, \mathfrak a, \mathfrak b\in\Mon(R)$ and $G(I)\subseteq K[X,Y]$. Let $\mathfrak a'$ and $\mathfrak b'$ be the ideals generated, respectively, by the sets of monomials $G(\mathfrak a)\cap K[X,Y]$ and $G(\mathfrak b)\cap K[X,Y]$. Then $I=\mathfrak a'\cdot\mathfrak b'$.
\end{corollary}
\begin{proof}
    Suppose that there exists $u\in G(\mathfrak a)\setminus K[X,Y]$. It follows that $X_k$ divides $u$ for some $k\in\llb 3,N\rrb$, whence $X_k$ also divides every monomial in $\{uv\colon v\in G(\mathfrak b)\}$. Since the monomials in $G(\mathfrak a\cdot\mathfrak b)=G(I)$ are not divisible by $X_k$, we have $uv\not\in\Min\{fg\colon f\in G(\mathfrak a),\ g\in G(\mathfrak b)\}$ for every $v\in G(\mathfrak b)$, whence
    \begin{align*}
        I &= \mathfrak a\cdot\mathfrak b \\
        &= \langle G(\mathfrak a\cdot\mathfrak b)\rangle \\
        &= \langle\Min\{fg\colon f\in G(\mathfrak a),\ g\in G(\mathfrak b)\}\rangle \\
        &= \langle\Min\{fg\colon f\in G(\mathfrak a)\setminus\{u\},\ g\in G(\mathfrak b)\}\rangle \\
        &=\langle G(\mathfrak a)\setminus\{u\} \rangle\cdot \langle G(\mathfrak b)\rangle\text{.}
    \end{align*}
    Repeating this argument for every $u\in G(\mathfrak a)\setminus K[X,Y]$ and every $v\in G(\mathfrak b)\setminus K[X,Y]$ yields the desired conclusion (note that $G(\mathfrak a)$ and $G(\mathfrak b)$ are finite sets, by Remark \ref{rem: mon min gen}).
\end{proof}

\begin{remark}\label{rem: mdeg0 = R}
    If $I\in\Mon(R)$ and $\mdeg(I)=0$, then $I=R$. Indeed, $I$ contains an element of min-degree zero, i.e., a polynomial $f\in R$ with $f(0)\ne 0$; since $I$ is generated by monomials and $f\in I$, necessarily $f(0)\in I$, whence $1\in I$.
\end{remark} 

For a monomial ideal $I \in \Mon(R)$, set
\[
B(I):=\{u\in G(I)\colon \deg(u)=\mdeg(I)\}.
\]
Let us recall from Section \ref{subsec: power monoids} the monoid homomorphism $\Phi\colon\mathcal P_{\rm fin}(\mathbb N)\to\Mon(R)$, sending a finite nonempty subset $A\subseteq \mathbb N$ to the  monomial ideal 
$I_A:=\langle X^{\max A-j}Y^j\colon j\in A  \rangle$.

We show in the next proposition that $\Phi$ preserves atoms.

\begin{proposition}\label{prop: Phi atoms to atoms}
The homomorphism $\Phi\colon\mathcal P_{\rm fin}(\mathbb N)\to\Mon(R), \; A \mapsto I_A$ sends atoms to atoms.
\end{proposition}
\begin{proof}
Let $A\in \mathcal P_{\rm fin}(\mathbb N)$ with $m=\max A$ and suppose that $I_A=\mathfrak a\cdot\mathfrak b$ with $\mathfrak a,\mathfrak b\in\Mon(R)$. Then
\[
B(I_A)=G(I_A)=\{X^{m-j}Y^{j}\colon j\in A\}.
\]
For monomial ideals, $B(\mathfrak a\cdot\mathfrak b)=B(\mathfrak a)B(\mathfrak b)$. Indeed, the products of generators of minimum total degree have total degree $\mdeg(\mathfrak a)+\mdeg(\mathfrak b)$, and every other product has larger degree.

Since $Y^m\in B(I_A)$, there exist $Y^p\in B(\mathfrak a)$ and $Y^q\in B(\mathfrak b)$ with $p+q=m$. Every monomial in $B(\mathfrak a)$ and $B(\mathfrak b)$ involves only $X$ and $Y$, and hence there are nonempty sets $M,N\subseteq\mathbb N$ such that
\[
B(\mathfrak a)=\{X^{p-k}Y^k\colon k\in M\},\qquad
B(\mathfrak b)=\{X^{q-k}Y^k\colon k\in N\}.
\]
The equality $B(I_A)=B(\mathfrak a)B(\mathfrak b)$ gives $A=M+N$. If $A$ is an atom, then $M=\{0\}$ or $N=\{0\}$. Since $p\in M$ and $q\in N$, this implies $p=0$ or $q=0$, and therefore $1\in\mathfrak a$ or $1\in\mathfrak b$. Thus one factor is $R$, and $I_A$ is an atom of $\Mon(R)$.
\end{proof}

\begin{example}\label{c4 atom Mon}
The set $\{0,1,2,4\}$ is an atom of $\mathcal P_{\rm fin}(\mathbb N)$. Hence, Proposition \ref{prop: Phi atoms to atoms} immediately gives that
\[
\mathfrak c_4=I_{\{0,1,2,4\}}=\langle X^4,X^3Y,X^2Y^2,Y^4\rangle
\]
is an atom of $\Mon(R)$. As recalled in Section \ref{sec: preliminaries}, $\mathfrak c_4$ is not an atom of $\I(R)$ when ${\rm char}\, K\neq 2$.
\end{example}

Next, we see to what extent \cite[Theorem 5.1]{Ge-Kh21} can be adapted to the monoid $\Mon(R)$. Let us define a further class of monomial ideals, namely $\mathfrak a_k={\langle X,Y\rangle}^k$, for any $k\in\mathbb N^+$.

\begin{theorem}\label{thm: 5.1 Mon}
The monoid of monomial ideals $\Mon(R)$ of the ring $R=K[X_1,\dots,X_N]$, where $K$ is a field and $N\ge 2$, is a $\mathsf{BF}$-monoid. Moreover,
\begin{enumerate}[label=\textup{(\arabic{*})}, mode=unboxed]
    \item\label{thm: 5.1 Mon 1} $\Mon(R)$ is neither locally finitely generated nor transfer Krull;
    \item\label{thm: 5.1 Mon 2} $\mathcal{U}_k(\Mon(R)) = \mathbb{N}_{\ge 2}$ for all $k\ge 2$;
    \item\label{thm: 5.1 Mon 3} $\mathsf{L}_{\Mon(R)}(\mathfrak a_k)=\llb 2,k \rrb$ for all $k \ge 2$;
    \item\label{thm: 5.1 Mon 4} $\Mon(R)$ is fully elastic.
\end{enumerate}
\end{theorem}
\begin{proof}
Let $I\in\Mon(R)$ with $m=\mdeg(I)$. Firstly, since every nonunit (in particular, every atom) of $\Mon(R)$ has positive min-degree by Remark \ref{rem: mdeg0 = R}, equation \eqref{eq: additivity min-deg} implies that the length of any factorization (into atoms of $\Mon(R)$) of $I$ is at most $m$. We will now show by induction on $m$ that $I$ always admits at least one factorization, thus proving that $\Mon(R)$ is a $\mathsf{BF}$-monoid.

If $m=0$, then $I=R$, by Remark \ref{rem: mdeg0 = R}. 
When $m\in\mathbb{N}^+$, either $I$ is itself an atom, or there exist $\mathfrak{a},\mathfrak{b}\in\Mon(R)\setminus\{R\}$ such that $I=\mathfrak{a}\cdot\mathfrak{b}$. By equation \eqref{eq: additivity min-deg} and Remark \ref{rem: mdeg0 = R}, we have $m=\mdeg(\mathfrak{a})+\mdeg(\mathfrak{b})$, $\mdeg(\mathfrak{a})\ge 1$ and $\mdeg(\mathfrak{b})\ge 1$. Thus, $\mdeg(\mathfrak{a}) < m$ and $\mdeg(\mathfrak{b}) < m$, whence by the inductive hypothesis both $\mathfrak a$ and $\mathfrak b$ can be factorized into atoms.

\ref{thm: 5.1 Mon 1} The fact that $\Mon(R)$ is not locally finitely generated follows from the proof of the corresponding statement for $\I(R)$ in \cite[Lemma 5.3.2]{Ge-Kh21}, since all the ideals involved are monomial, and a monomial ideal that is an atom in $\I(R)$ is clearly also an atom in $\Mon(R)$. For the second part of the statement, assume to the contrary that there are a Krull monoid $H$ and a transfer homomorphism $\theta\colon\Mon(R)\to H$. Consider the monomial ideals
\[\mathfrak a_1 = \langle X,Y\rangle,\ \mathfrak a_2 = {\langle X,Y\rangle}^2={\mathfrak a_1}^2,\ \mathfrak b_2 = \langle X^2,Y^2\rangle.\]
A straightforward computation shows that $\mathfrak a_1\mathfrak a_2 = \mathfrak a_1\mathfrak b_2$, whence $\theta(\mathfrak a_1)\theta(\mathfrak a_2)=\theta(\mathfrak a_1)\theta(\mathfrak b_2)$ in $H$. Since Krull monoids are cancellative, we infer $\theta(\mathfrak a_2)=\theta(\mathfrak b_2)$. By Lemma \ref{lem: transfer Krull}, for any $I\in\Mon(R)$ we have $I\in\mathscr A(\Mon(R))$ if and only if $\theta(I)\in\mathscr A(H)$. Since $\mathfrak b_2$ is an atom of $\Mon(R)$, it follows that $\theta(\mathfrak b_2)$ is an atom of $H$. Thus, $\theta(\mathfrak a_2)$ is also an atom of $H$, whence $\mathfrak a_2$ is an atom of $\Mon(R)$, a contradiction.

\ref{thm: 5.1 Mon 2} One can proceed exactly as in \cite[Lemma 5.3.3]{Ge-Kh21} to show that $\mathcal U_2(\Mon(R))=\mathbb N_{\ge 2}$, as all the ideals occurring in that proof are monomial. Since $\Mon(R)$ is a $\mathsf{BF}$-monoid, the conclusion follows by \cite[Lemma 3.1.3]{Ge-Kh21}.

\ref{thm: 5.1 Mon 3} Clearly, $\mdeg(\mathfrak a_k)=k$ and, for $k\ge 2$, the ideal $\mathfrak a_k$ is not an atom of $\Mon(R)$. Since the length of a factorization of a monomial ideal is bounded by its min-degree and $1\in \mathsf L_{\Mon(R)}(\mathfrak a_k)$ if and only if $\mathfrak a_k$ is an atom (i.e., if and only if $k=1$), we infer that $\mathsf L_{\Mon(R)}(\mathfrak a_k)\subseteq\llb 2,k\rrb$, for any $k\ge 2$. For the converse inclusion we refer to the proof of \cite[Theorem 5.1.3]{Ge-Kh21}, noting that all the ideals therein are monomial, with one exception: when proving that $2\in \mathsf L_{\I(R)}(\mathfrak a_5)$, the authors use a factorization into genuine polynomial ideals to get around the fact that $\mathfrak c_4$ is generally not an atom of $\I(R)$. In our case, by considering the identity
\[\mathfrak a_5 = {\langle X,Y\rangle}^5 = \langle X,Y\rangle \langle X^4, X^3Y, X^2Y^2, Y^4\rangle = \mathfrak a_1 \mathfrak c_4,\]
involving only monomial ideals, and recalling Example \ref{c4 atom Mon}, we can conclude that $2\in \mathsf L_{\Mon(R)}(\mathfrak a_5)$.

\ref{thm: 5.1 Mon 4} Set $\mathfrak p=\langle X\rangle$, denote by $H_1$ the free monoid with basis $\{\mathfrak p\}$ and let $H_2=\{\mathfrak a\in\Mon(R)\colon \mathfrak p\nmid\mathfrak a\}$. The ideal $\mathfrak p$ is principal and therefore invertible in the monoid of nonzero fractional ideals of $R$. In particular, it is a cancellative element of this monoid, and hence also a cancellative element of $\Mon(R)$. Moreover, $\mathfrak p$ is clearly prime, whence we obtain $\Mon(R) = H_1\times H_2$. Now observe that $\mathfrak a_k\in H_2$ for all $k\in\mathbb N^+$, as $Y^k\in\mathfrak a_k$. By item \ref{thm: 5.1 Mon 3}, we can then apply \cite[Prop.~3.2]{Ge-Kh21} to conclude that $\Mon(R)$ is fully elastic.
\end{proof}

\subsection{Atoms and sets of lengths in $\Mon(R)$} 

We continue our investigation of $\Mon(R)$ by constructing new classes of atoms in this monoid. Fix a positive integer $n\ge 3$ and $a_1,\dots,a_n,a_{n+1}\in \mathbb{N}^+$ verifying conditions \ref{cond_1}, $a_{n+1}=a_1+\dots+a_{n-1}+2a_n$, and \ref{cond_2}, $a_{i+1}>2(a_1+\dots+a_i)$ for all $i\in\llb 1,n-1\rrb$. Keeping the notation of Section \ref{subsect: I_B}, for every nonempty $I\subseteq\llb 1,n+1\rrb$ we set $a_I=\sum_{i\in I} a_i$,  while setting $a_\emptyset=0$. For $r\in \llb3, n\rrb$ let us define:
\begin{equation}\label{def tilde b_r}
    \tilde{\mathfrak b}_r:=\left\langle
   \mathfrak b_{a_1}\cdot\prod_{i=3}^r \mathfrak{b}_{a_i},\,
   X^{a_{\llb 3,r\rrb}-a_2}Y^{a_3-a_2}
   \right\rangle.
\end{equation}

Note that $X^{a_{\llb 3,r\rrb} - a_2}Y^{a_{3} - a_2}$ is a proper generator of $\Tilde{\mathfrak{b}}_{r}$, i.e., $X^{a_{\llb 3,r\rrb} - a_2}Y^{a_{3} - a_2}\notin \mathfrak b_{a_1}\cdot\prod_{i=3}^r \mathfrak{b}_{a_i}$. To show that this is true, we first observe that
\[\mathfrak b_{a_1}\cdot\prod_{i=3}^r \mathfrak{b}_{a_i}=\langle X^{a_1+a_{\llb 3,r\rrb}-a_J} Y^{a_{J}} \colon J\subseteq\{1\}\cup\llb3,r\rrb\rangle,\]
and, by Lemma \ref{lem: mon_gen}, if $X^{a_{\llb 3,r\rrb} - a_2}Y^{a_{3} - a_2}\in \mathfrak b_{a_1}\cdot\prod_{i=3}^r \mathfrak{b}_{a_i}$, one of the generators of $\mathfrak b_{a_1}\cdot\prod_{i=3}^r \mathfrak{b}_{a_i}$ must divide it. Observe that for $J\subseteq \{1\} \cup\llb 3,r\rrb$
\begin{equation}\label{a1-a_3}
a_J>a_1 \Rightarrow a_J\ge a_3.
\end{equation}
Thus, if there exists $J\subseteq\{1\}\cup\llb3,r\rrb$ such that $X^{a_1+a_{\llb 3,r\rrb}-a_J} Y^{a_{J}}$ divides $X^{a_{\llb 3,r\rrb} - a_2}Y^{a_{3} - a_2}$, then it must be $a_1+a_{\llb 3,r\rrb}-a_J\le a_{\llb 3,r\rrb} - a_2$, which implies $a_J\ge a_1+a_2$ and, by \eqref{a1-a_3}, $a_J\ge a_{3}$. Moreover, we must also have $a_J\le a_3-a_2$, but this contradicts the previous inequality. We can conclude that
\[ G(\tilde{\mathfrak{b}}_r) = \{ X^{a_1+a_{\llb 3,r\rrb}-a_J} Y^{a_{J}} \colon J\subseteq\{1\}\cup\llb3,r\rrb\}\cup\{X^{a_{\llb 3,r\rrb} - a_2}Y^{a_{3} - a_2}\}. \]

\begin{proposition}\label{lem: eq-monomial}
Let $n\ge 3$ be a positive integer. The following hold for every $r\in \llb3, n\rrb$:

\vspace{.2cm}
\begin{enumerate*}[label=\textup{(\arabic{*})}, resume, mode=unboxed]
    \item\label{lem: eq-monomial(1)} $\mathfrak{b}_{a_2}\cdot \tilde{\mathfrak{b}}_{r}= \prod_{i=1}^r\mathfrak{b}_{a_i}$.
\end{enumerate*}

\vspace{.2cm}
\begin{enumerate*}[label=\textup{(\arabic{*})}, resume, mode=unboxed]
    \item\label{lem: eq-monomial(2)} $I_{C_n} = \mathfrak{b}_{a_2}\cdot \tilde{\mathfrak{b}}_{r}\cdot \prod_{i=r+1}^{n+1}\mathfrak{b}_{a_i}$.
\end{enumerate*}
\end{proposition}

\begin{proof}
\ref{lem: eq-monomial(1)} By definition $\mathfrak{b}_{a_i}=\langle X^{a_i}, Y^{a_i}\rangle$ and $\Tilde{\mathfrak{b}}_{r} = \langle \mathfrak b_{a_1}\cdot\prod_{i=3}^r \mathfrak{b}_{a_i}, X^{a_{\llb 3,r\rrb} - a_2}Y^{a_{3} - a_2} \rangle$, whence 
\[\mathfrak{b}_{a_2}\cdot \tilde{\mathfrak{b}}_{r}= \langle\prod_{i=1}^r\mathfrak{b}_{a_i}, X^{a_{\llb3,r\rrb}}Y^{a_3-a_2},X^{a_{\llb3,r\rrb}-a_2}Y^{a_3}\rangle.\]
It is easy to see that $X^{a_{\llb3,r\rrb}}Y^{a_1+a_2}, X^{a_{\llb1,r\rrb}-a_3}Y^{a_3}\in \prod_{i=1}^r\mathfrak{b}_{a_i}$. But then
\[X^{a_{\llb3,r\rrb}}Y^{a_3-a_2}=X^{a_{\llb3,r\rrb}}Y^{a_1+a_2}\cdot Y^{a_3-2a_2-a_1} \in \prod_{i=1}^r\mathfrak{b}_{a_i},\]
since $a_3-2a_2-a_1>0$ by condition \ref{cond_2}, and 
\[X^{a_{\llb3,r\rrb}-a_2}Y^{a_3}=X^{a_{\llb1,r\rrb}-a_3}Y^{a_3}\cdot X^{a_3-2a_2-a_1} \in \prod_{i=1}^r\mathfrak{b}_{a_i}.\]
Therefore, $\mathfrak{b}_{a_2}\cdot \tilde{\mathfrak{b}}_{r}= \prod_{i=1}^r\mathfrak{b}_{a_i}$ as claimed.

\vspace{.2cm}
\ref{lem: eq-monomial(2)} It follows from item \ref{lem: eq-monomial(1)} and Remark \ref{rem: I_C as product of b_i's}.
\end{proof}

The relevance of the above proposition is made clear by the following result, whose proof is provided in Section \ref{sect: proofs}.
\begin{theorem}\label{th: b_r atom}
Let $n\ge 3$ be a fixed integer. For every $r\in \llb3, n\rrb$ ($n\ge 3$), $\tilde{\mathfrak{b}}_{r}$ is an atom of $\Mon(R)$.
\end{theorem}

The construction of the class of atoms $\tilde{\mathfrak{b}}_r$ of $\Mon(R)$ gives us the opportunity to answer negatively, for this monoid, the question posed at the end of Section \ref{sec: atoms in I(R)}.
\begin{corollary}\label{cor: lengths}
For every fixed integer $n\ge 2$, $\mathsf{L}_{\Mon(R)}(I_{C_n})=\llb 2, n+1\rrb$.
\end{corollary}
\begin{proof}
    Fix $n\ge 2$. By Remark \ref{rem: I_C as product of b_i's}, we obtain $\{2,n+1\}\subseteq\mathsf{L}_{\Mon(R)}(I_{C_n})$, as the factorizations therein are into monomial ideals which are known to be atoms of $\I(R)$, and so also of $\Mon(R)$. By Theorem \ref{thm: I_B atom}\ref{thm: I_B atom(2)}, we have $\mathsf{L}_{\I(R)}(I_{C_n})\subseteq\llb 2, n+1\rrb$. Since any factorization $z\in\mathsf Z_{\Mon(R)}(I_{C_n})$ can be expanded into a factorization $w\in\mathsf Z_{\I(R)}(I_{C_n})$ with $\lvert w\rvert\ge \lvert z\rvert$, we obtain $\sup\mathsf L_{\I(R)}(I_{C_n})\ge \sup\mathsf L_{\Mon(R)}(I_{C_n})$, whence $\mathsf L_{\Mon(R)}(I_{C_n})$ is also contained in $\llb 2,n+1\rrb$ ($1\not\in \mathsf L_{\Mon(R)}(I_{C_n})$ as $I_{C_n}$ is not an atom). If $n=2$, the conclusion is trivial; if $n\ge 3$, it follows from Proposition \ref{lem: eq-monomial}\ref{lem: eq-monomial(2)} and Theorem \ref{th: b_r atom}.
\end{proof}

We conclude with a further theorem on the atomicity of $2$-generated ideals $\langle X^n,Y^m\rangle$. The proof is given in Section \ref{sect: proofs}.

\begin{theorem}\label{xm yn mon}
For every $m,n\in \mathbb{N}^+$, the ideal $I=\langle X^m, Y^n\rangle$ is an atom in $\Mon(R)$. 
\end{theorem}

\subsection{Proofs of Theorems \ref{th: b_r atom} and \ref{xm yn mon}}\label{sect: proofs}

We start our discussion by proving some technical lemmas.
\begin{lemma}\label{lemma a_J-a_I}
Fix a positive integer $n\ge 4$. Let $I,J\subseteq \{1\}\cup \llb 3,r\rrb$ with $r\in \llb 4,n\rrb$. If $a_J-a_I>a_1$ and $a_J-a_I\ne a_3-a_1$, then $a_J-a_I\ge a_3$.
\end{lemma}
\begin{proof}
We first observe that $a_J-a_I=a_{J\setminus I}-a_{I\setminus J}$. If $I\setminus J=\emptyset$, then $a_J-a_I=a_{J\setminus I}$ and $a_{J\setminus I}>a_1$ trivially implies $a_{J\setminus I}\ge a_3$, as $J\setminus I\subseteq \{1\}\cup \llb 3,r\rrb$. Assume then $I\setminus J\ne\emptyset$. Clearly, our assumptions also imply $J\setminus I\ne \emptyset$ (otherwise $a_J-a_I<0$), so there exist $s,t>0$ such that $s=\max (J\setminus I)$ and $t=\max (I\setminus J)$. Note that, since $(J\setminus I)\cap (I\setminus J)=\emptyset$, $s\ne t$.
Moreover, it must be $s>t$. In fact, if $s<t$ (implying that $t\ge 3$), then 
\[a_{I\setminus J}\ge a_t>a_{\llb 1,t-1\rrb}\ge a_{\llb 1,s\rrb}\ge a_{J\setminus I},\]
but this is absurd since we are assuming $a_J-a_I=a_{J\setminus I}-a_{I\setminus J}>0$.
We can also prove that $s\ge 4$. If this is not the case, in fact, the only admissible alternative is $s=3$ and $t=1$. However, this leads to $I\setminus J=\{1\}$ and $J\setminus I=\{3\}$, i.e., to $a_J-a_I=a_3-a_1$, a case that we discarded by assumption. 

Therefore, we can conclude that 
\[a_J-a_I=a_{J\setminus I}-a_{I\setminus J}\ge a_s-a_{\llb 1,s-1\rrb}>a_{\llb 1,s-1\rrb}\ge a_{\llb 1,3\rrb}>a_3.\]
\end{proof}

\begin{lemma}\label{lem: min-deg}
Let $I,  \mathfrak{a}, \mathfrak{b} \in \Mon(R)\setminus{\{R\}}$ with $\mdeg(I) = m$, $\mdeg(\mathfrak{a}) = d$, $\mdeg(\mathfrak{b}) = e$, $X^m, Y^m \in I$ and $I =  \mathfrak{a}\cdot \mathfrak{b}$. Then:

\vspace{.2cm}
\begin{enumerate*}[label=\textup{(\arabic{*})}, resume, mode=unboxed]
\item\label{lem: min-deg(1)} $d, e \geq 1$.
\end{enumerate*}

\vspace{.2cm}
\begin{enumerate*}[label=\textup{(\arabic{*})}, resume, mode=unboxed]
\item\label{lem: min-deg(2)} $X^d,Y^d\in G(\mathfrak a)$ and $X^e,Y^e\in G(\mathfrak b)$.
\end{enumerate*}
\end{lemma}

\begin{proof}
\ref{lem: min-deg(1)} Apply Remark \ref{rem: mdeg0 = R}.

\vspace{.2cm}
\ref{lem: min-deg(2)} Since $m=\mdeg(I)$ and $X^m,Y^m\in I$, by Lemma \ref{lem: mon_gen} we deduce $X^m,Y^m\in G(I)$. As $G(I)=G(\mathfrak a\cdot\mathfrak b)\subseteq\{fg\colon f\in G(\mathfrak a),g\in G(\mathfrak b)\}$ and $X^m\in G(I)$, necessarily $X^i\in G(\mathfrak a)$ and $X^j\in G(\mathfrak b)$ for some $i,j$ such that $i+j=m$. Since $\mdeg(\mathfrak a)=d$ and $\mdeg(\mathfrak b)=e$, we have $i\ge d$ and $j\ge e$, whence equation \eqref{eq: additivity min-deg} (i.e., $m=d+e$) implies $i=d$ and $j=e$. Analogously, we can prove $Y^d\in G(\mathfrak a)$ and $Y^e\in G(\mathfrak b)$.
\end{proof}

Let $n\ge 3$ and $r\in\llb 3,n\rrb$. Recall that
\begin{equation}\label{b_r generators}
\Tilde{\mathfrak{b}}_{r} = \langle \left\{ X^{a_1+a_{\llb 3,r\rrb}-a_J} Y^{a_{J}} \colon J\subseteq\{1\}\cup\llb3,r\rrb\right\}, X^{a_{\llb3,r\rrb}-a_2}Y^{a_3-a_2}\rangle.
\end{equation}
Note that $\mdeg(\Tilde{\mathfrak{b}}_{r})=a_1+a_{\llb 3,r\rrb}=a_{\{1\}\cup\llb 3,r\rrb}$. In fact, by definition of the $a_i$'s, $a_1<a_3-2a_2$ and so $a_1+a_{\llb 3,r\rrb}<a_{\llb 3,r\rrb}+a_3-2a_2$.

\begin{lemma}\label{lem: b_r}
Assume $\Tilde{\mathfrak{b}}_{r}=\mathfrak{a}\cdot\mathfrak{b}$
for some $\mathfrak{a},\mathfrak{b}\in \Mon(R)\setminus \{R\}$ and let $\alpha=\mdeg(\mathfrak{a})$ and $\beta=\mdeg(\mathfrak{b})$. Then:

\vspace{.2cm}
\begin{enumerate*}[label=\textup{(\arabic{*})}, resume, mode=unboxed]
    \item\label{lem: b_r(1)} $\alpha,\beta\ge1$ and $a_{\{1\}\cup\llb 3,r\rrb}=\alpha+\beta$;
\end{enumerate*}

\vspace{.2cm}
\begin{enumerate*}[label=\textup{(\arabic{*})}, resume, mode=unboxed]
    \item\label{lem: b_r(2)} $X^\alpha,Y^\alpha\in G(\mathfrak{a})$ and $X^\beta,Y^\beta\in G(\mathfrak{b})$;
\end{enumerate*}

\vspace{.2cm}
\begin{enumerate*}[label=\textup{(\arabic{*})}, resume, mode=unboxed]
    \item\label{lem: b_r(3)} $\alpha=a_{I_1}$, $\beta=a_{I_2}$ with $I_1,I_2\subseteq\{1\}\cup \llb 3,r\rrb$ and $I_1\cap I_2=\emptyset$.
\end{enumerate*}
\end{lemma}

\begin{proof}
\ref{lem: b_r(1)} Follows from Lemma \ref{lem: min-deg}\ref{lem: min-deg(1)} and equation \eqref{eq: additivity min-deg}.

\vspace{.2cm}
\ref{lem: b_r(2)} Follows from Lemma \ref{lem: min-deg}\ref{lem: min-deg(2)}.

\vspace{.2cm}
\ref{lem: b_r(3)} By Lemma \ref{lem_5.6}\ref{lem_5.6(ii)}, we get $
(\Tilde{\mathfrak{b}}_{r})_K[a_{\{1\}\cup\llb 3,r\rrb}]=\mathfrak{a}_K[\alpha]\cdot\mathfrak{b}_K[\beta]$,
where
\begin{equation}\label{span}
{(\Tilde{\mathfrak{b}}_{r})}_K[a_{\{1\}\cup\llb 3,r\rrb}]= \spa_K\left\{ X^{a_1+a_{\llb 3,r\rrb}-a_J} Y^{a_{J}} \colon J\subseteq\{1\}\cup\llb3,r\rrb\right\}.
\end{equation}
By \ref{lem: b_r(2)}, $X^\alpha,Y^\alpha \in \mathfrak{a}_K[\alpha]$ and $X^\beta,Y^\beta \in \mathfrak{b}_K[\beta]$. Hence, $X^{\alpha}Y^{\beta} \in {(\Tilde{\mathfrak{b}}_{r})}_K[a_{\{1\}\cup\llb 3,r\rrb}]$, and so
\[
X^{\alpha}Y^{\beta} = X^{a_1+a_{\llb 3,r\rrb}-a_{I_2}} Y^{a_{I_2}}
\]
for some $I_2\subseteq\{1\}\cup\llb 3,r\rrb$. This implies that $\alpha = a_1+a_{\llb 3,r\rrb}-a_{I_2} = a_{I_1}$ for some $I_1\subseteq\{1\}\cup\llb 3,r\rrb$ with $I_1 \cap I_2 = \emptyset$ and $\beta=a_{I_2}$.
\end{proof}

We now have all the ingredients to finally prove Theorem \ref{th: b_r atom}.

\begin{proof}[Proof of Theorem \ref{th: b_r atom}]
Fix $n\ge 3$ and let $r\in \llb 3,n\rrb$. Assume as a contradiction that $\tilde{\mathfrak{b}}_r=\mathfrak{a}\cdot\mathfrak{b}$
for some $\mathfrak{a},\mathfrak{b}\in \Mon(R)\setminus \{R\}$. By Corollary \ref{cor: wlog mon factors}, we can assume $G(\mathfrak a),G(\mathfrak b)\subseteq K[X,Y]$ without loss of generality. By Lemma \ref{lem: b_r}, $X^\alpha, Y^\alpha\in G(\mathfrak a)$ and $X^\beta, Y^\beta\in G(\mathfrak b)$ with $\alpha=\mdeg(\mathfrak{a})=a_I$ and $\beta=\mdeg(\mathfrak{b})=a_{\{1\}\cup\llb 3,r\rrb}-\alpha=a_{(\{1\}\cup\llb 3,r\rrb)\setminus I}$, where $I$ is a nonempty proper subset of ${\{1\}\cup\llb 3,r\rrb}$.
By symmetry, we can also assume without loss of generality that $\alpha\le \beta$. A rough estimate then ensures that $\alpha< a_r$. In fact, if $\alpha\ge a_r$, then $a_{\{1\}\cup\llb 3,r\rrb}=\alpha+\beta\ge 2\alpha\ge 2a_r$, which is impossible, as $a_r>a_{\llb 1,r-1\rrb}$.
Thus, we can assume without loss of generality that, if $r\ge 4$, then $\alpha=a_I$ with $I\subseteq \{1\}\cup\llb 3,r-1\rrb$, while, if $r=3$, then $\alpha=a_1$. As a consequence, $\beta\ge a_r$.

\medskip

CASE 0: $\mathfrak a=\langle X^\alpha, Y^\alpha\rangle$, $\mathfrak b=\langle X^\beta, Y^\beta\rangle$

In this case, we have
\begin{equation}\label{case 0}
\tilde{\mathfrak{b}}_r=\mathfrak a\cdot\mathfrak b=\langle X^{a_{\{1\}\cup\llb 3,r\rrb}}, X^\alpha Y^\beta,X^\beta Y^\alpha, Y^{a_{\{1\}\cup\llb 3,r\rrb}}\rangle\text{.}
\end{equation}
Since $\alpha = a_I$ for some $I\subseteq {\{1\}\cup\llb 3,r\rrb}$ and $\beta = a_{{\{1\}\cup\llb 3,r\rrb}\setminus{I}}$, the remarks following equation \eqref{def tilde b_r} imply that none of the generators in \eqref{case 0} divides $X^{a_{\llb 3,r\rrb}-a_2}Y^{a_3-a_2}\in\tilde{\mathfrak{b}}_r$, contradicting Lemma \ref{lem: mon_gen}.

\medskip

Now we distinguish two further cases depending on whether $\alpha$ is equal or not to $a_1$, and for each case we consider several subcases according to the structure of $\mathfrak a$ and $\mathfrak b$.

\medskip

CASE 1: $\alpha=a_1$ and $\beta=a_{\llb 3,r\rrb}$. In particular, $X^{a_1},Y^{a_1}\in G(\mathfrak a)$ and $X^{a_{\llb 3,r\rrb}},Y^{a_{\llb 3,r\rrb}}\in G(\mathfrak b)$.

\medskip

Subcase I: $\mathfrak a=\langle X^{a_1}, Y^{a_1}\rangle$

Since $X^{a_{\llb 3,r\rrb}-a_2}Y^{a_3-a_2}\in \tilde{\mathfrak{b}}_r$, by Lemma \ref{lem: mon_gen} at least one of the minimal generators of $\mathfrak a\cdot\mathfrak b$ must divide it. As in CASE 0, there must exist $k, \ell\in\mathbb N$ such that $X^k Y^\ell\in G(\mathfrak b)$ and either $X^{k+a_1}Y^\ell \mid X^{a_{\llb 3,r\rrb}-a_2}Y^{a_3-a_2}$ or $X^k Y^{\ell+a_1} \mid X^{a_{\llb 3,r\rrb}-a_2}Y^{a_3-a_2}$. Note that $k+\ell \ge \mdeg(\mathfrak b) = a_{\llb 3, r\rrb}$ and, since $X^{a_{\llb 3,r\rrb}},Y^{a_{\llb 3,r\rrb}}\in G(\mathfrak b)$, necessarily $1\le k, \ell < a_{\llb 3, r\rrb}$.

Assume that $X^{k+a_1}Y^\ell \mid X^{a_{\llb 3,r\rrb}-a_2}Y^{a_3-a_2}$.
Therefore, 
\begin{align}
   \label{k 1I 1} k &\le a_{\llb 3,r\rrb}-a_2-a_1,\\
   \label{l 1I 1} \ell&\le a_3-a_2.
\end{align}
We claim that none of the minimal generators of $\tilde{\mathfrak b}_r$ divides $X^k Y^{\ell+a_1}$, implying $X^k Y^{\ell+a_1}\in{(\mathfrak a\cdot\mathfrak b)\setminus\tilde{\mathfrak b}_r}$, a contradiction. To prove our claim, we first notice that, since $k< a_{\llb 3,r\rrb}-a_2$, we have $X^{a_{\llb 3,r\rrb}-a_2}Y^{a_3-a_2}\nmid X^k Y^{\ell+a_1}$. Moreover, if
\begin{equation}\label{1I 1}
  X^{a_{\{1\}\cup\llb 3,r\rrb}-a_J}Y^{a_J}\mid X^{k}Y^{\ell+a_1} \text{\ for some\ } J\subseteq \{1\}\cup \llb 3,r\rrb,  
\end{equation}
then, using \eqref{k 1I 1}, we get $a_{\{1\}\cup\llb 3,r\rrb}-a_J\le k \le a_{\llb 3,r\rrb}-a_2-a_1$, implying $a_J\ge 2a_1+a_2$ and so $a_J\ge a_3$, by \eqref{a1-a_3}. But \eqref{1I 1} also implies $a_J\le \ell+a_1$ and, by \eqref{l 1I 1}, we reach the contradiction $a_J\le a_3-a_2 + a_1 < a_3$.

On the other hand, assume that $X^{k}Y^{\ell+a_1} \mid X^{a_{\llb 3,r\rrb}-a_2}Y^{a_3-a_2}$.
Therefore, 
\begin{align}
   \label{k 1I 2} k &\le a_{\llb 3,r\rrb}-a_2,\\
   \label{l 1I 2} \ell&\le a_3-a_2-a_1.
\end{align}
We claim that none of the minimal generators of $\tilde{\mathfrak b}_r$ divides $X^{k+a_1} Y^{\ell}$, leading to the contradiction $X^{k+a_1} Y^{\ell}\notin \tilde{\mathfrak b}_r$. To prove this claim, we first notice that, since $\ell< a_3-a_2$, we have $X^{a_{\llb 3,r\rrb}-a_2}Y^{a_3-a_2}\nmid X^{k+a_1} Y^{\ell}$.  Moreover, if
\begin{equation}\label{1I 2}
  X^{a_{\{1\}\cup\llb 3,r\rrb}-a_J}Y^{a_J}\mid X^{k+a_1}Y^{\ell} \text{\ for some\ } J\subseteq \{1\}\cup \llb 3,r\rrb,  
\end{equation}
then, using \eqref{k 1I 2}, we get $a_{\{1\}\cup\llb 3,r\rrb}-a_J\le k+a_1 \le a_{\{1\}\cup\llb 3,r\rrb}-a_2$, implying $a_J\ge a_2$ and so $a_J\ge a_3$, by \eqref{a1-a_3}. But \eqref{1I 2} also implies $a_J\le \ell$ and, by \eqref{l 1I 2}, we reach the contradiction $a_J\le a_3-a_2 - a_1 < a_3$.

\medskip

Subcase II: $\mathfrak a\supsetneq\langle X^{a_1}, Y^{a_1}\rangle$ 

As $X^{a_1}, Y^{a_1}\in G(\mathfrak a)$, there must exist $i,j\in \mathbb{N}$ such that $i+j\ge {a_1}$, $1\le i,j<a_1$ and $X^i Y^j\in G(\mathfrak a)$. We claim that none of the minimal generators of $\tilde{\mathfrak b}_r$ divides $X^{i+a_{\llb 3,r\rrb}}Y^{j}$, leading to the contradiction $X^{i+a_{\llb 3,r\rrb}}Y^{j}\in {(\mathfrak a\cdot\mathfrak b)\setminus\tilde{\mathfrak b}_r}$. To prove our claim, we first notice that, since $j< a_1<a_3-a_2$, we have $X^{a_{\llb 3,r\rrb}-a_2}Y^{a_3-a_2}\nmid X^{i+a_{\llb 3,r\rrb}}Y^{j}$. Moreover, if $X^{a_{\{1\}\cup\llb 3,r\rrb}-a_J}Y^{a_J}\mid X^{i+a_{\llb 3,r\rrb}}Y^{j}$ for some $J\subseteq \{1\}\cup \llb 3,r\rrb$, then we get $a_J\le j< a_1$, implying $a_J=0$, whence $a_{\{1\}\cup\llb 3,r\rrb}-a_J=a_{\{1\}\cup\llb 3,r\rrb}> i+ a_{\llb 3,r\rrb}$, a contradiction.

\medskip

CASE 2: $\alpha\ne a_1$, i.e., $\alpha\ge a_3$ by Lemma \ref{lem: b_r}\ref{lem: b_r(3)}. As we already observed in the initial part of the proof, under our assumptions this can happen only if $r\ge 4$ (in particular, $n\ge 4$). As a consequence, we also have $\beta\le a_{\{1\}\cup \llb 4,r\rrb}$, where the set $\llb 4,r\rrb$ is nonempty. Recall that $X^\alpha,Y^\alpha\in G(\mathfrak a)$ and $X^\beta,Y^\beta\in G(\mathfrak b)$.

\medskip

Subcase I: $\mathfrak a=\langle X^\alpha, Y^\alpha\rangle$

As $X^{a_{\llb 3,r\rrb}-a_2}Y^{a_3-a_2}\in \tilde{\mathfrak{b}}_r$, one of the minimal generators of $\mathfrak a\cdot\mathfrak b$ must divide it. As in CASE 0, there must exist $X^k Y^\ell\in G(\mathfrak b)$ (in particular, $k+\ell\ge \beta$ and $1\le k,\ell < \beta$) such that either $X^{k+\alpha}Y^\ell \mid X^{a_{\llb 3,r\rrb}-a_2}Y^{a_3-a_2}$ or $X^k Y^{\ell+\alpha} \mid X^{a_{\llb 3,r\rrb}-a_2}Y^{a_3-a_2}$. We can immediately discard the latter case: since $\alpha\ge a_3$ and $\ell\ge 1$, we always have $\ell+\alpha>a_3>a_3-a_2$.

Therefore, we must have $X^{k+\alpha}Y^\ell \mid X^{a_{\llb 3,r\rrb}-a_2}Y^{a_3-a_2}$, i.e.,
\begin{align}
   \label{k 2I} k &\le a_{\llb 3,r\rrb}-a_2-\alpha,\\
   \label{l 2I} \ell&\le a_3-a_2.
\end{align}
We claim that none of the minimal generators of $\tilde{\mathfrak b}_r$ divides $X^k Y^{\ell+\alpha}$, leading to the contradiction $X^k Y^{\ell+\alpha}\in(\mathfrak a\cdot\mathfrak b)\setminus \tilde{\mathfrak b}_r$. To prove this claim we first notice that, since $k< a_{\llb 3,r\rrb}-a_2$, we have $X^{a_{\llb 3,r\rrb}-a_2}Y^{a_3-a_2}\nmid X^k Y^{\ell+\alpha}$.  Moreover, if 
\begin{equation}\label{2I}
  X^{a_{\{1\}\cup\llb 3,r\rrb}-a_J}Y^{a_J}\mid X^{k}Y^{\ell+\alpha} \text{\ for some\ } J\subseteq \{1\}\cup \llb 3,r\rrb,  
\end{equation}
then, using \eqref{k 2I}, we get $a_{\{1\}\cup\llb 3,r\rrb}-a_J\le k \le a_{\llb 3,r\rrb}-a_2-\alpha$, implying that $a_J\ge a_1+a_2+\alpha$. Recalling that $\alpha=a_I$ for some $I\subseteq\{1\}\cup \llb 3,r\rrb $ and that $r\ge 4$, by Lemma \ref{lemma a_J-a_I} we have two possibilities: either $a_J=a_3-a_1+\alpha$ or $a_J\ge a_3+\alpha$. But \eqref{2I} and \eqref{l 2I} also imply $a_J\le \ell+\alpha\leq a_3-a_2+\alpha$ and in both the above cases we reach a contradiction.

\medskip

Subcase II: $\mathfrak b=\langle X^\beta, Y^\beta\rangle$

Arguing as in the previous subcase, there must exist $X^i Y^j\in G(\mathfrak a)$ such that $i+j\ge\alpha$, $1\le i,j<\alpha$ and  $X^{i+\beta}Y^j\mid X^{a_{\llb 3,r\rrb}-a_2}Y^{a_3-a_2}$. Indeed, since $j+\beta>a_r>a_3>a_3-a_2$, we can discard the case $X^{i}Y^{j+\beta}\mid X^{a_{\llb 3,r\rrb}-a_2}Y^{a_3-a_2}$. Thus,
\begin{align}
   \label{i 2II} i &\le a_{\llb 3,r\rrb}-a_2-\beta,\\
   \label{j 2II} j&\le a_3-a_2.
\end{align}
We claim that none of the minimal generators of $\tilde{\mathfrak b}_r$ divides $X^i Y^{j+\beta}$, leading to the contradiction $X^i Y^{j+\beta}\in(\mathfrak a\cdot\mathfrak b)\setminus\tilde{\mathfrak b}_r$. To prove our claim, we first notice that, since $i< a_{\llb 3,r\rrb}-a_2$, we have $X^{a_{\llb 3,r\rrb}-a_2}Y^{a_3-a_2}\nmid X^i Y^{j+\beta}$.  Moreover, if 
\begin{equation}\label{2II}
  X^{a_{\{1\}\cup\llb 3,r\rrb}-a_J}Y^{a_J}\mid X^{i}Y^{j+\beta} \text{\ for some\ } J\subseteq \{1\}\cup \llb 3,r\rrb,  
\end{equation}
then, using \eqref{i 2II}, we get $a_{\{1\}\cup\llb 3,r\rrb}-a_J\le i \le a_{\llb 3,r\rrb}-a_2-\beta$, implying $a_J\ge a_1+a_2+\beta$. Since $\beta=a_I$ for some $I\subseteq\{1\}\cup \llb 3,r\rrb$ and $r\ge 4$, using Lemma \ref{lemma a_J-a_I}, \eqref{2II} and \eqref{j 2II}, we find a contradiction as in the previous subcase.

\medskip

Subcase III: $\mathfrak a \supsetneq \langle X^\alpha,Y^\alpha\rangle$, $\mathfrak b \supsetneq \langle X^\beta,Y^\beta\rangle$

One of the minimal generators of $\mathfrak a\cdot\mathfrak b$ must divide $X^{a_{\llb 3,r\rrb}-a_2}Y^{a_3-a_2}\in \tilde{\mathfrak{b}}_r$. In view of subcases I and II, there must exist $i,j,k,\ell\in\mathbb N$ such that $X^i Y^j\in G(\mathfrak a)$, $i+j\ge \alpha$, $1\le i,j<\alpha$, $X^k Y^\ell\in G(\mathfrak b)$, $k+\ell\ge \beta$, $1\le k,\ell <\beta$ and $X^{i+k}Y^{j+\ell}\mid X^{a_{\llb 3,r\rrb}-a_2}Y^{a_3-a_2}$. Thus,
\begin{align}
   \label{i+k 2III} i+k &\le a_{\llb 3,r\rrb}-a_2,\\
   \label{j+l 2III} j+\ell &\le a_3-a_2.
\end{align}

We claim that at least one of $X^{i+\beta}Y^j$ or $X^{k}Y^{\ell+\alpha}$ does not belong to $\tilde{\mathfrak b}_r$, contradicting $\tilde{\mathfrak b}_r=\mathfrak a\cdot\mathfrak b$.

To prove this claim, suppose that $X^{i+\beta}Y^j\in \tilde{\mathfrak{b}}_r$. One of the minimal generators of $\tilde{\mathfrak{b}}_r$ must therefore divide it. Inequality \eqref{j+l 2III} implies $j<a_3-a_2$, whence $X^{a_{\llb 3,r\rrb}-a_2}Y^{a_3-a_2}\nmid X^{i+\beta}Y^j$. Moreover, if $J\subseteq \{1\}\cup\llb 3,r\rrb$ is such that $X^{a_{\{1\}\cup\llb 3,r\rrb}-a_J}Y^{a_J}\mid X^{i+\beta}Y^{j}$, then $a_J\le j< a_3-a_2$ and this can happen only if $a_J=a_1$ (if $a_J=0$, then $a_{\{1\}\cup\llb 3,r\rrb}-a_J=a_{\{1\}\cup\llb 3,r\rrb}=\alpha+\beta>i+\beta$). Thus, $a_{\llb 3,r\rrb}=a_{\{1\}\cup\llb 3,r\rrb}-a_J\le i+\beta$, i.e., $i\ge a_{\llb 3,r\rrb}-\beta$ and, from \eqref{i+k 2III}, we get $a_{\llb 3,r\rrb}-\beta+k\le i+k\le  a_{\llb 3,r\rrb}-a_2 $. As we observed at the beginning of CASE 2 that $\beta\le a_{\{1\}\cup \llb 4,r\rrb}$, from the previous chain of inequalities we obtain $k\le \beta-a_2 \le a_{\{1\}\cup \llb 4,r\rrb}-a_2<a_{\llb 4,r\rrb}$. Moreover, we derive from \eqref{j+l 2III} that $\ell+\alpha <a_3-a_2+\alpha$. 

We conclude the proof of the claim by showing that none of the minimal generators of $\tilde{\mathfrak{b}}_r$ divides $X^{k}Y^{\ell+\alpha}\notin \tilde{\mathfrak{b}}_r$. We first note that $X^{a_{\llb 3,r\rrb}-a_2}Y^{a_3-a_2}\nmid X^{k}Y^{\ell+\alpha}$ since $a_{\llb 3,r\rrb}-a_2>a_{\llb 4,r\rrb}>k$. Moreover, if
\begin{equation}\label{2III}
  X^{a_{\{1\}\cup\llb 3,r\rrb}-a_{\bar{J}}}Y^{a_{\bar{J}}}\mid X^{k}Y^{\ell+\alpha} \text{\ for some\ } \bar J\subseteq \{1\}\cup \llb 3,r\rrb,  
\end{equation}
then $a_{\{1\}\cup\llb 3,r\rrb}-a_{\bar{J}}\le k \le \beta-a_2 = a_{\{1\}\cup \llb 3,r\rrb}-\alpha-a_2$, implying $a_{\bar{J}}\ge a_{2}+\alpha$. Then, by Lemma \ref{lemma a_J-a_I} we have two possibilities: either $a_{\bar{J}}=a_3-a_1+\alpha$ or $a_{\bar{J}}\ge a_3+\alpha$. Since \eqref{2III} also implies $a_{\bar{J}}\le \ell+\alpha< a_3-a_2+\alpha$, both of these possibilities lead to a contradiction.

All possible cases have been examined and discarded, so we can conclude that $\tilde{\mathfrak{b}}_r\ne \mathfrak a\cdot \mathfrak b$ for all $\mathfrak a, \mathfrak b\in \Mon(R)\setminus \{R\}$.
\end{proof}

We end this section by proving the atomicity in $\Mon(R)$ of $2$-generated ideals of the form $\langle X^n,Y^m\rangle$.

\begin{lemma}\label{lem: 2-gen}
Let $m,n\in \mathbb{N}^+$ such that $m<n$, and let $I\in\Mon(R)$ with $X^m, Y^n\in G(I)$ and $\mdeg(I)=m$.
Assume $I=\mathfrak{a}\cdot \mathfrak{b}$ for some $\mathfrak{a},\mathfrak{b}\in \Mon(R)\setminus\{R\}$. Then:

\vspace{.2cm}
\begin{enumerate*}[label=\textup{(\arabic{*})}, resume, mode=unboxed]
    \item\label{lem: 2-gen(1)} $m=\alpha+\beta$, with $\alpha=\mdeg(\mathfrak{a})\ge 1$ and $\beta=\mdeg(\mathfrak{b})\ge 1$.
\end{enumerate*}

\vspace{.2cm}
\begin{enumerate*}[label=\textup{(\arabic{*})}, resume, mode=unboxed]
    \item\label{lem: 2-gen(2)} there exist $a,b\in \mathbb{N}^+$, with $a\ge \alpha$, $b\ge \beta$, and $a+b=n$ such that $X^\alpha, Y^a\in\ G(\mathfrak{a})$ and $X^\beta, Y^b\in\ G(\mathfrak{b})$.
\end{enumerate*}
\end{lemma}
\begin{proof}
\ref{lem: 2-gen(1)} This follows from equation \eqref{eq: additivity min-deg} and Remark \ref{rem: mdeg0 = R}.

\vspace{.2cm}
\ref{lem: 2-gen(2)} Since $G(I) = G(\mathfrak a\cdot\mathfrak b) \subseteq \{fg\colon f\in G(\mathfrak a),\ g\in G(\mathfrak b)\}$, necessarily $X^i,Y^a\in G(\mathfrak a)$ and $X^j,Y^b\in G(\mathfrak b)$ for some $i\in\llb \alpha,m\rrb,\ a\in\llb \alpha,n\rrb,\ j\in\llb \beta,m\rrb,\ b\in\llb \beta,n\rrb$ such that $i+j = m$ and $a+b = n$. It follows that $i=\alpha$ and $j=\beta$.
\end{proof}

\begin{remark}
It is worth noting that when $I$ is a monomial ideal such that $X^m,Y^n\in\ G(I)$ for some $m,n\in\mathbb{N}^+$ and $G(I)\subseteq K[X,Y]$, in order to check whether $I$ is an atom one needs only to compute a finite number of products $\mathfrak{a}\cdot\mathfrak{b}$. Indeed, if $I=\mathfrak{a}\cdot\mathfrak{b}$ (with $\mathfrak a,\mathfrak b\in\Mon(R)\setminus\{R\}$), we can assume $G(\mathfrak a), G(\mathfrak b)\subseteq K[X,Y]$ by Corollary \ref{cor: wlog mon factors}. Arguing as in the proof of Lemma \ref{lem: 2-gen}\ref{lem: 2-gen(2)}, there exist $\alpha,\beta\in\llb 1,m\rrb,\ a, b\in\llb 1,n\rrb$ such that $X^\alpha, Y^a\in G(\mathfrak a)$, $X^\beta, Y^b\in G(\mathfrak b)$, $\alpha+\beta = m$ and $a+b=n$. By minimality of $G(\mathfrak a)$, every monomial $X^\gamma Y^c\in G(\mathfrak a)$ satisfies $\gamma\le \alpha$ and $c\le a$, and similarly every monomial $X^\delta Y^d\in G(\mathfrak b)$ satisfies $\delta\le \beta$ and $d\le b$.
\end{remark}

\begin{proof}[Proof of Theorem \ref{xm yn mon}]
We already know from \cite[Prop.~5.10.1]{Ge-Kh21} that $I=\langle X^m,Y^n\rangle$ is an atom of $\I{(R)}$, hence also of $\Mon{(R)}$, when $m=n$. Therefore, by symmetry, it is enough to consider the case $m<n$. If $m=1$, equation \eqref{eq: additivity min-deg} ensures the atomicity of $I$ in $\Mon(R)$: if $I=\mathfrak{a}\cdot \mathfrak{b}$ for some $\mathfrak{a}, \mathfrak{b}\in \Mon(R)$, one among $\mathfrak{a}$ and $\mathfrak{b}$ has min-degree zero and must therefore be the whole ring $R$.
For $m\ge 2$, assume as a contradiction that there exist proper nonzero monomial ideals of $R$, say $\mathfrak a, \mathfrak b$, such that $I=\mathfrak{a}\cdot \mathfrak{b}$. By Lemma \ref{lem: 2-gen}, we have that $X^\alpha, Y^a\in G(\mathfrak a)$ and $X^\beta, Y^b\in G(\mathfrak b)$, where $\alpha=\mdeg(\mathfrak{a})$, $\beta=\mdeg(\mathfrak{b})$, $\alpha+\beta=m$, $a\ge \alpha\ge 1$, $b\ge \beta\ge 1$ and $a+b=n$. It follows that $X^\alpha Y^b\in I$, but this is a contradiction since $m>\alpha$ and $n>b$, so neither of the generators of $I$ divides it.
\end{proof}

\section{Prospects for future research}\label{sec: conclusion}
   The results of this paper suggest several directions for future research, as they provide large families of atoms in $\I(R)$ and $\Mon(R)$, along with techniques to construct them. In light of the recent breakthrough by Reinhart \cite{Rein25} on the arithmetic of $\mathcal P_{\rm fin, 0}(\mathbb N)$, we expect that evaluating finer arithmetical invariants of these monoids, such as the (monotone) catenary degree (see \cite{Ge-Re19} for the definition and some examples in which it can be computed), should be feasible. Moreover, in view of Theorem \ref{thm: I_B atom}\ref{thm: I_B atom(2)} and Corollary \ref{cor: lengths}, it would be interesting to attempt the construction of ideals whose sets of lengths are not intervals. Such results could, in turn, yield information about the sets of distances in the monoids under consideration. Another possible line of research would consist in detecting the \emph{strong atoms} of $\I(R)$ or $\Mon(R)$; in the terminology of \cite{Sophie24}, these are atoms $a$ of a monoid $H$ for which $\mathsf Z_H(a^n)$ is a singleton for every $n\in\mathbb N$. In this context, the recent paper \cite{Rosi25} by Rath and Rissner on generating sets of powers of monomial ideals may serve as a valuable reference. Finally, we highlight the problem of computing the density of atoms in the aforementioned monoids of ideals. On this topic, it is worth recalling that in the reduced power monoid $\mathcal P_{\rm fin, 0}(\mathbb N)$, almost every element is an atom \cite[Theorem 6.1]{Bi-Ge22}.

\end{document}